\documentclass[mathpazo,online,leqno]{jml}

\renewcommand\thisvolume{1}
\renewcommand\thisnumber{2}
\renewcommand\thisyear {2022}
\renewcommand\thismonth{June}

\renewcommand\originalreceiveddate{10/12/2023}

\renewcommand\acceptdate{11/09/2024}

\renewcommand\pagestart{1}
\renewcommand\pageend{xx}

\renewcommand\articledoi{doi.org/10.4208/jml.xxx}

\renewcommand\handler{Jiequn Han}

\usepackage{lipsum}

\usepackage{textcomp}
\usepackage{sidecap}

\usepackage[hypertexnames=false]{hyperref}
\hypersetup{
    colorlinks=true,
    citecolor=blue,
    linkcolor=blue,
    urlcolor=teal,
    breaklinks=true
}
\makeatletter
\renewcommand\section{\@startsection{section}{1}{\z@}%
           {25\p@ \@plus 6\p@ \@minus 3\p@}%
           {10\p@ \@plus 6\p@ \@minus 3\p@}%
           {\fontsize{13pt}{0cm}\selectfont\bfseries\boldmath}}

\renewcommand\subsection{\@startsection{subsection}{2}{\z@}%
           {13\p@ \@plus 6\p@ \@minus 3\p@}%
           {6\p@ \@plus 6\p@ \@minus 3\p@}%
           {\fontsize{12pt}{0cm}}}
\renewcommand\subsubsection{\@startsection{subsubsection}{3}{\z@}%
           {12\p@ \@plus 6\p@ \@minus 3\p@}%
           {\p@}%
           {\normalfont\normalsize}}
\renewcommand{\paragraph}[1]{%
  \par
  \addvspace{\medskipamount}
  \noindent
  \textbf{#1\@addpunct{.}}\enspace\ignorespaces
}
\makeatother
\usepackage[x11names,dvipsnames]{xcolor}
\usepackage{graphicx}
\usepackage{url}
\usepackage[doi=false,isbn=false, sortcites=true,eprint=true, date=year, style=numeric,uniquename=init,giveninits=true,sorting=nty,url=false,maxnames=99,backref=true,date=year]{biblatex} 
\usepackage{filecontents}

\DeclareFieldFormat*{volume}{\mkbibbold{#1}} 
\DeclareFieldFormat*{title}{\mkbibquote{#1\adddot}}  
\DeclareFieldFormat*{booktitle}{\mkbibemph{#1}}  
\DeclareSourcemap{
  \maps[datatype=bibtex]{
    \map{
      \step[fieldsource=edition,
            match=\regexp{^(\d+)\.\s+(ed.|edition)$},
            replace=\regexp{$1}]
    }
  }
}
\DeclareFieldFormat*{journaltitle}{\mkbibemph{#1}} 
\DeclareFieldFormat[article,Article,book,inbook,incollection,inproceedings,patent,thesis,unpublished,misc]{year}{#1} %
\renewbibmacro{in:}{} %
\renewbibmacro*{issue+date}{%
  \printfield{issue}%
  \setunit*{\addspace}%
  \usebibmacro{date}%
  \newunit} %
\NewBibliographyString{noeditor}
\DefineBibliographyStrings{english}{%
  noeditor      =   {s\adddot ed\adddot},
}

\renewbibmacro*{editor+others}{%
  \ifuseeditor
    {\ifnameundef{editor}
      {\bibstring{noeditor}}
      {\printnames{editor}%
       \setunit{\addcomma\space}%
       \usebibmacro{editor+othersstrg}%
       \clearname{editor}}}
    {}}

\renewbibmacro*{editor}{%
  \ifuseeditor
    {\ifnameundef{editor}
      {\bibstring{noeditor}}
      {\printnames{editor}%
       \setunit{\addcomma\space}%
       \usebibmacro{editorstrg}%
       \clearname{editor}}}
    {}}
\usepackage{amssymb}
\usepackage{amsfonts}
\usepackage{mathrsfs}
\usepackage[capitalize,noabbrev]{cleveref}
\usepackage{amsthm}
\expandafter\let\expandafter\oldproof\csname\string\proof\endcsname
\let\oldendproof\endproof
\renewenvironment{proof}[1][\proofname]{%
  \oldproof[\bfseries\itshape #1] 
}{\oldendproof}
\makeatletter

\makeatletter
\renewenvironment{proof}[1][\proofname]{\par
  \pushQED{\qed}%
  \normalfont \topsep6\p@\@plus6\p@\relax
  \trivlist
  \item\relax
  {\itshape
  #1\@addpunct{.}}\hspace\labelsep\ignorespaces
}{%
  \popQED\endtrivlist\@endpefalse
}
\makeatother

\theoremstyle{plain}
\newtheorem{theorem}{Theorem}[section]
\newtheorem{proposition}[theorem]{Proposition}
\newtheorem{lemma}[theorem]{Lemma}
\newtheorem{corollary}[theorem]{Corollary}
\newtheorem{claim}[theorem]{Claim}

\theoremstyle{definition}
\newtheorem{definition}[theorem]{Definition}
\newtheorem{problem}[theorem]{Problem}
\newtheorem{assumption}[theorem]{Assumption}

\theoremstyle{remark}
\newtheorem{remark}[theorem]{Remark}

\usepackage{array}
\usepackage{latexsym}
\usepackage{subfiles}
\usepackage{amsmath}
\allowdisplaybreaks[4]
\usepackage{stackengine}
\usepackage{empheq}
\usepackage{physics}
\renewcommand*{\eval}[1]{\left.{#1}\right|}
\usepackage{here}
\usepackage{longtable}
\usepackage{lipsum}
\usepackage{caption}
\usepackage{capt-of}
\usepackage{threeparttable}

\usepackage{svg}
\usepackage{enumitem}
\setlist[description]{%
  topsep=10pt,               
  itemsep=5pt,               
  font={\bfseries\rmfamily}, 
}

\usepackage{mathtools}
\usepackage{morewrites}
\usepackage{etoolbox}

\DeclareFontFamily{U}{matha}{\hyphenchar\font45}
\DeclareFontShape{U}{matha}{m}{n}{
<-6> matha5 <6-7> matha6 <7-8> matha7
<8-9> matha8 <9-10> matha9
<10-12> matha10 <12-> matha12
}{}
\DeclareSymbolFont{matha}{U}{matha}{m}{n}

\DeclareFontFamily{U}{mathx}{\hyphenchar\font45}
\DeclareFontShape{U}{mathx}{m}{n}{
<-6> mathx5 <6-7> mathx6 <7-8> mathx7
<8-9> mathx8 <9-10> mathx9
<10-12> mathx10 <12-> mathx12
}{}
\DeclareSymbolFont{mathx}{U}{mathx}{m}{n}

\DeclareMathDelimiter{\vvvert} {0}{matha}{"7E}{mathx}{"17}%

\DeclarePairedDelimiterX{\normiii}[1]
{\vvvert}
{\vvvert}
{\ifblank{#1}{\:\cdot\:}{#1}}
\usepackage{appendix}

\DeclareSymbolFont{EulerExtension}{U}{euex}{m}{n}
\DeclareMathSymbol{\euintop}{\mathop} {EulerExtension}{"52}
\DeclareMathSymbol{\euointop}{\mathop} {EulerExtension}{"48}

\numberwithin{equation}{section}

\allowdisplaybreaks[4]

\usepackage{makecell}
\usepackage{tikz}
\usetikzlibrary{calc}
\tikzset{
    ncbar angle/.initial=90,
    ncbar/.style={
        to path=(\tikztostart)
        -- ($(\tikztostart)!#1!\pgfkeysvalueof{/tikz/ncbar angle}:(\tikztotarget)$)
        -- ($(\tikztotarget)!($(\tikztostart)!#1!\pgfkeysvalueof{/tikz/ncbar angle}:(\tikztotarget)$)!\pgfkeysvalueof{/tikz/ncbar angle}:(\tikztostart)$)
        -- (\tikztotarget)
    },
    ncbar/.default=0.5cm,
}

\tikzset{square left brace/.style={ncbar=0.5cm}}
\tikzset{square right brace/.style={ncbar=-0.5cm}}

\tikzset{round left paren/.style={ncbar=0.5cm,out=120,in=-120}}
\tikzset{round right paren/.style={ncbar=0.5cm,out=60,in=-60}}
\usetikzlibrary{positioning, fit}   
\tikzset{block/.style={draw, thick, text width=2cm ,minimum height=1.3cm, align=center},   
line/.style={-latex}     
}  
\usepackage{boldline}
\usepackage{multirow}
\usepackage{here}
\usepackage[super]{nth}

\crefname{theorem}{Theorem}{Theorems}
\crefname{definition}{Definition}{Definitions}
\crefname{problem}{Problem}{Problems}
\crefname{fact}{Fact}{Facts}
\crefname{proposition}{Proposition}{Propositions}
\crefname{lemma}{Lemma}{Lemmas}
\crefname{corolary}{Corolary}{Corolaries}
\crefname{assumption}{Assumption}{Assumptions}
\crefname{claim}{Claim}{Claims}

\crefname{remark}{Remark}{Remarks}
\crefname{example}{Example}{Examples}
\crefname{corollary}{Corollary}{Corollaries}
\crefname{subsubsection}{Subsection}{Sections}
\crefname{subsection}{Subsection}{Sections}
\crefname{section}{Section}{Sections}
\crefname{chapter}{Chapter}{Chapters}
\crefname{table}{Table}{Tables}
\crefname{figure}{Figure}{Figures}
\crefname{algorithm}{Algorithm}{Algorithms}

\renewcommand{\ref}{\cref}

\newcommand{\Pc}{\mathcal{P}_\mathrm{c}}
\newcommand{\Pcal}{\mathcal{P}}

\newcommand{\Ccinf}{C_c^\infty}
\newcommand{\R}{\mathbb{R}}
\newcommand{\Z}{\mathbb{Z}}
\newcommand{\Y}{\mathcal{Y}}
\newcommand{\la}{\left\langle}
\newcommand{\ra}{\right\rangle}
\newcommand{\intRd}{\int_{\R^d}}
\newcommand{\intRdY}{\int_{\R^d\times\Y}}

\DeclareMathOperator{\Div}{\mathrm{div}}

\DeclareMathOperator*{\Lip}{\mathrm{Lip}}
\DeclareMathOperator*{\supp}{\mathrm{supp}}
\DeclareMathOperator*{\minimize}{\mathrm{minimize}}
\DeclareMathOperator*{\diam}{\mathrm{diam}}

\DeclareMathOperator{\Id}{\mathrm{Id}}

\def\Set#1{\Setdef#1\Setdef}
\def\Setdef#1|#2\Setdef{\left\{#1\,\;\mathstrut\vrule\,\;#2\right\}}%
\renewcommand{\th}{%
    \ifmmode
        ^\mathrm{th}%
    \else%
        \textsuperscript{th}\xspace%
    \fi%
}
\makeatletter
\newcommand{\subalign}[1]{%
  \vcenter{%
    \Let@ \restore@math@cr \default@tag
    \baselineskip\fontdimen10 \scriptfont\tw@
    \advance\baselineskip\fontdimen12 \scriptfont\tw@
    \lineskip\thr@@\fontdimen8 \scriptfont\thr@@
    \lineskiplimit\lineskip
    \ialign{\hfil$\m@th\scriptstyle##$&$\m@th\scriptstyle{}##$\hfil\crcr
      #1\crcr
    }%
  }%
}
\makeatother

\usepackage{tikz}
\usetikzlibrary{matrix,chains,positioning,decorations.pathreplacing,arrows}
\usetikzlibrary{shapes,arrows}
\tikzset{
	font={\fontsize{9pt}{12}\selectfont}}

\begin{document}

\title[Formulations and Existence results for ODE-Net]{{Variational formulations of ODE-Net} {as a mean-field optimal control problem} \mbox{and existence results}}

\author[1]{
Noboru Isobe
	\thanks{
	Corresponding author.
{mail:\tt\href{mailto:nobo0409@g.ecc.u-tokyo.ac.jp}{nobo0409@g.ecc.u-tokyo.ac.jp}}, Supported by {JSPS KAKENHI 22J20130}.
	}
}
\author[2]{
Mizuho Okumura
	\thanks{
		{\tt\href{mailto:okumura.mizuho.p3@gmail.com}{okumura.mizuho.p3@gmail.com}}.
	}
}

\affil[1]{Graduate School of Mathematical Sciences, The University of Tokyo, Tokyo, Japan.}
\affil[2]{Graduate School of Science, Tohoku University, Sendai, Japan.}

\begin{abstract}
This paper presents a mathematical analysis of ODE-Net, a continuum model of deep neural networks (DNNs). 
In recent years, Machine Learning researchers have introduced ideas of replacing the deep structure of DNNs with ODEs as a continuum limit. 
These studies regard the ``learning'' of ODE-Net as the minimization of a ``loss'' constrained by a parametric ODE. 
Although the existence of a minimizer for this minimization problem needs to be assumed, only a few studies have investigated the existence analytically in detail. 
In the present paper, the existence of a minimizer is discussed based on a formulation of ODE-Net as a measure-theoretic mean-field optimal control problem. 
The existence result is proved when a neural network describing a vector field of ODE-Net is linear with respect to learnable parameters. 
The proof employs the measure-theoretic formulation combined with the direct method of Calculus of Variations. 
Secondly, an idealized minimization problem is proposed to remove the above linearity assumption. 
Such a problem is inspired by a kinetic regularization associated with the Benamou--Brenier formula and universal approximation theorems for neural networks. 
\end{abstract}

%
%
%
\keywordone{Deep Learning,}
\keywordtwo{ResNet,}
\keywordthree{ODE-Net,}
\keywordfour{Benamou--Brenier formula,}
\keywordfive{Mean-Field Game.}

\maketitle

\section{Introduction}
Deep Neural Networks~(DNNs), or Deep Learning, now constitute a core of artificial intelligence technology, but their theoretical inner mechanisms have yet to be explored.
In particular, there have been few theoretical contributions regarding ``learning'' DNNs, despite practical demands for them, where ``learning'' is, broadly speaking, to minimize the so-called ``loss'' by optimizing a parameter $\theta$ of DNNs.  

Our research aims to establish a well-posed mathematical formulation of the learning.
To achieve this aim, some researchers have brought languages of dynamical systems and differential equations into DNNs, for example, in \cite{Haber17,E2017,sonoda2017double}. 
In short, one can regard a continuum limit of DNNs in their depth as an ODE.
Many researchers have attempted to dissect DNNs through some ODEs, designated as \emph{ODE-Net} throughout the paper.
For more information on these attempts, see the survey in \cref{subsec:background}.
Based on this survey, well-posednesses, such as the existence of a minimizer of loss, have not yet been fully explored in the context of these studies.

Accordingly, our goal in this paper is to prove the existence of a minimizer for learning ODE-Net, formulated as a regularized minimization problem constrained by a continuity equation.

\subsection{Target problems and main results}\label{subsec:mainresults}
First of all, we are going to study the existence of a minimizer of the following \emph{kinetic-regularized} minimization problem:
\begin{problem}[Kinetic regularized learning problem constrained by ODE-Net]\label{prob:kinetic_prob}
Let $\lambda \geq 0$ and $\epsilon >0$ be constants, let $\Y$ be a subset of $\R^d$ and let $v\colon\R^d\times\R^m\to\R^d$ and $\ell\colon\R^d\times\Y\to\R_+$ be continuous. 
Let $\mu_0\in\mathcal{P}_c(\R^d\times\Y)$ be a given training data.
Set 
\begin{align}\label{eq:J_kin} 
J(\mu,\theta)
\coloneqq
\int_{\mathbb{R}^d\times\Y}\ell\dd{\mu_T}
+\int_0^T\int_{\mathbb{R}^d\times\Y}\qty({\frac{\lambda}{2}\abs{v\qty(x,\theta_t)}^{2}}
+\frac{\epsilon}{2}\abs{\theta_t}^2)\dd{\mu_t(x,y)\dd t}
\end{align}
for $\mu\in C(\qty[0,T];(\mathcal{P}_2(\R^d\times\mathcal{Y}),W_2))$ and $\theta\in L^2\qty(0,T;\R^m)$.
Note that $v(\bullet,\theta)\in L^2(\dd{\mu})$ is a vector field on $\R^d$ for $\mu\in\Pc(\R^d)$ and $\theta\in\R^m$.
The learning problem constrained by ODE-Net is posed as the following constrained minimization problem:
\begin{gather} 
    \notag 
    \inf \Set{ J(\mu,\theta) | \mu\in C\qty(\qty[0,T];(\mathcal{P}_2(\R^d\times\mathcal{Y}),W_2)),\; \theta\in L^2\qty(0,T;\R^m)  }, 
    \intertext{subject to } 
    \left\{
\begin{aligned}
        \partial_t\mu_t+\Div_x(\mu_t(x,y)v(x,\theta_t))&=0,\ \qty(x,y)\in \R^d\times\Y,\ t\in (0,T),\\
    \left.\mu_t\right|_{t = 0}&=\mu_0,       
\end{aligned}
\right.\label{eq:ODE2}
\end{gather}
where $\mathcal{P}_c(\R^d\times\mathcal{Y})$ denotes the set of regular and Borel probability measures compactly supported on $\R^d\times\mathcal{Y}$, $(\mathcal{P}_2(\R^d\times\mathcal{Y}),W_2)$ denotes the ($L^2$-)Wasserstein space defined in \cref{subsec:Wass}, $C(\qty[0,T];(\mathcal{P}(\R^d\times\mathcal{Y}),W_2))$ denotes the set of curves which is continuous with respect to the Wasserstein topology (see also \cref{def:measure}), and \[\mu\in C([0,T]; (\mathcal{P}_2(\R^d\times \mathcal{Y}),W_2))\] is supposed to solve the equation~\eqref{eq:ODE2} in the distributional sense of~\cref{def:wk_sol}.
\end{problem}

\begin{remark}
In \cref{prob:kinetic_prob}, the ODE-Net corresponds to the continuity equation \eqref{eq:ODE2} with a parameter $\theta_t$, and the learning to the minimization of a functional $J$ with respect to a parameter $\theta_t$ and a solution $\mu_t$ to ODE~\eqref{eq:ODE2}.
\end{remark}

The first term in~\eqref{eq:J_kin} measures the so-called \emph{loss}.
The second term in \eqref{eq:J_kin} is called a ``kinetic regularization'' in \cite{pmlr-v119-finlay20a} because it represents the kinetic energy when $v(\bullet,\theta)\  (\theta\in\R^m)$ is regarded as a velocity field on $\R^d$. 
By letting this kinetic energy be as small as possible, we could control the velocity field so that the support of the solution $\mu_t$ to~\eqref{eq:ODE2} does not change wildly. 
The third term is often called an $L^2$-regularization, which is familiar with the well-known Ridge regression.

In order to prove existence of a minimizer for \cref{prob:kinetic_prob}, we shall impose the following assumptions on $\Y$, $\ell$ and $v$ :

\begin{assumption}\label{assump:label_loss}
The label set $\mathcal{Y}\subset\R^d$ is compact, and the loss function $\ell\colon\R^d\times\Y\to\R_{+}$ is a continuous function of $2$-growth, see also \cref{def:measure}.
\end{assumption}

In addition, following the previous works on ODE-Net \cite{pmlr-v139-sander21a,sander2022do,Scagliotti22,barboni2021global,yang2023tensor}, we impose the assumption below that the neural network $v(x,\theta)$ is linear with respect to $\theta$, but not necessarily linear with respect to $x$.

\begin{assumption}\label{assump:linear_NN}
The neural network $v$ in \eqref{eq:ODE2} is linear with respect to $\theta$, i.e., the parameter $\theta$ is a $d\times p$ matrix and $v$ satisfies 
\begin{equation}
    v(x,\theta)=\theta f(x),\label{eq:linear_NN}
\end{equation}
where $f\colon\R^d\to\R^p$ is a Lipschitz continuous function.
\end{assumption}

\cref{assump:linear_NN} is not a serious restriction. 
In fact, \cite[Theorem 1]{vialard20shooting} shows that for a neural network $v$ that is nonlinear with respect to $\theta$, there exists another neural network that is linear with respect to $\theta$ and can approximate the solution $\mu$ of ODE-Net \eqref{eq:ODE2}.
Thus, \cref{assump:linear_NN} is not so restrictive in discussing the existence of the minimizer for \cref{prob:kinetic_prob}.
Rather, \cref{assump:linear_NN} can address neural networks unbounded with respect to parameters $\theta$, which commonly appear in modern DNNs.
In contrast, previous theoretical works often assume a bounded neural network, details of which will be given in \cref{subsec:background} below.

Under these assumptions, we obtain one of our main results in the present paper.
\begin{theorem}[Existence of a minimizer]\label{thm:existence_of_kineticODE_Net}
Under \cref{assump:label_loss,assump:linear_NN}, there exists a minimizer $(\mu,\theta)\in C([0,T];(\mathcal{P}_{2}(\R^d\times\Y),W_2))\times L^2\qty(0,T;\R^m)$ for \cref{prob:kinetic_prob}.
\end{theorem}
\noindent
It should be noted that by virtue of this theorem, one can assume that the deep learning model as in \cref{prob:kinetic_prob} with \cref{assump:label_loss,assump:linear_NN} is well-defined so that we can pursue the mathematical analysis of the learning of ODE-Nets. We also remark here that the uniqueness of such minimizers cannot be generally expected since the problem is over-determined with a large degree of freedom in $\theta$. We will also mention the uniqueness in \cref{rmk:unique} below.

We note that 
\cref{assump:linear_NN} does not hold for all neural networks.
For example, two-layer ReLU networks $v(x,\theta)=A(Bx)_+$, $\theta=(A,B)$, $A,B\in\R^{d\times d}$, are not linear with respect to $\theta$. This network is quite commonly used, not only in ODE-Net but also in the so-called ResNet, as illustrated in \cite[Figure 2]{resnet}.

In order to provide existence results for these cases as well, we shall consider an ideal or relaxed version of \cref{prob:kinetic_prob}.
To this end, we shall employ the universal approximation theorem by Cybenko~\cite{Cybenko1989} or the Kolmogorov--Arnol'd representation theorem shown by \cite{Kolmogorov57,Arnold57,Sprecher65}; they insist that neural networks $v$ can approximate or represent arbitrary vector fields. Those theorems inspire that the ODE-Net is no longer parametrized by $\theta$, i.e., the ODE-Net is just driven by a family of vector fields $(v_t)_{t\in[0,T]}$.
From this perspective, our ideal setting for the learning reads:

\begin{problem}[Ideal learning problem]\label{prob:ideal_prob}
Let $\lambda > 0$ be a strictly positive constant, let $\ell\colon\R^d\times\Y\to\R_+$ be continuous, and let $\mu_0\in\Pc(\R^d\times\Y)$ be a given input data.
Set 
\begin{align}
\widehat{J}(\mu,v)
\coloneqq
\int_{\mathbb{R}^d\times\Y}\ell\dd{\mu_T}
+\int_0^T\int_{\mathbb{R}^d\times\Y} {\frac{\lambda}{2}\abs{v\qty(x,t)}^2}\dd{\mu_t(x,y)\dd t} \label{eq:J_ideal}
\end{align}
for $\mu\in C(\qty[0,T];(\mathcal{P}_2(\R^d\times\mathcal{Y}),W_2))$ and $v\in L^2\qty(\dd{\mu_t\dd t})$, 
where $v\in L^2\qty(\dd{\mu_t\dd t})$ means that the squared integral of $v(x,t)$ in the measure $\dd \mu_t(x)$ over $\R^d$ is integrable in time $t$ over $ [0,T]$.
Then an \emph{ideal} learning problem constrained by ODE-Net is posed as the following constrained minimization problem:
\begin{gather}
\inf \Set{
\widehat{J}(\mu,v) | 
\mu\in C(\qty[0,T];(\mathcal{P}_2(\R^d\times\mathcal{Y}),W_2)),\; 
v\in L^2\qty(\dd{\mu_t\dd t})
}\nonumber,
\intertext{subject to}
\left\{
\begin{aligned}
    \partial_{t} \mu_{t}+\Div_x\left(v_t\mu_{t}\right)&=0\text{ (in the sense of \cref{def:wk_sol})},\\ 
    \left.\mu_{t}\right|_{t=0}&=\mu_{0}.
\end{aligned}
\right.\label{eq:ODE3}
\end{gather}
\end{problem}
\noindent
In contrast to \cref{prob:kinetic_prob}, where the parameter $\theta$ is a variable to the functional $J$, the vector field $v$ itself is a variable to the functional $\widehat{J}$ in \cref{prob:ideal_prob}.
For this idealized problem containing a broader class of vector fields, we also establish the existence of a minimizer as in the following theorem:

\begin{theorem}[Existence of a minimizer]\label{thm:ideal_prob}
Under \cref{assump:label_loss}, there exists a minimizer $\mu\in C([0,T];(\mathcal{P}_{2}(\R^d\times\Y),W_2))$ and $v\in L^2(\dd\mu_t\dd t)$ for \cref{prob:ideal_prob}.
\end{theorem}

We have been discussing the well-posedness of ``learning'' of ODE-Net by formulating it via a \emph{mean-field} optimal control problem, in the sense that we have to control trajectories in the space of probability measures $\mu_t$. Through the above discussion, it is suggested that such a learning framework successfully gives a mathematical way to analyze the learning processes of DNNs.
Our main results obtained in this analysis are interesting from the viewpoint of the Calculus of Variations in that minimizers exist for nonlinear optimal control problems such as \cref{prob:kinetic_prob,prob:ideal_prob}. 
Moreover, the proofs of our theorems will ensure that every minimizing sequence contains a convergent subsequence in a suitable topology, leading to the well-posedness of sequential minimization algorithms such as Gradient Descent (GD).

\subsection{Contributions of the paper}\label{subsec:aim}
The present paper contributes to establishing the existence of a minimizer under situations where the regularization parameter $\lambda$ is not necessarily large in \cref{thm:existence_of_kineticODE_Net}.
This situation can be addressed because, in contrast to the paper \cite{bonnet2022measure}, we use an argument that does not rely on a strong convexity of $J$ to prove the existence of a minimizer.
In addition, this theorem can apply to unbounded and non-differentiable neural networks $v(x,\theta)$, which are important targets in applications.

As a comparison, a key to our convergence results is to obtain the existence of a minimizer of both $\mu$ and $\theta$ under reasonable assumptions.
The authors in \cite{bonnet2022measure} required strong convexity of $J$, or a sufficiently large parameter $\lambda$, in order to obtain strong compactness.
In addition, Thorpe and Gennip~\cite{thorpe2020deep} and  Esteve et al.~\cite{esteve2021largetime} obtained existence results under an $H^1$-regularization of $\theta_t$, and Herty et al.~\cite{herty2022large} under boundedness for the Lipschitz constant of $\theta\colon[0,T]\to\R^m$; broadly speaking, both of them are assuming that the ``differentials'' of the parameters $\theta_t$ in time are controlled.
While these studies are novel in that they do not impose assumptions on regularization parameters such as $\lambda$, the assumptions of the continuity or differentiability on parameters $\theta\colon[0,T]\to\R^m$ should be relaxed or removed because the functions that ODE-Net can approximate are limited and the expected value of $\ell$ cannot be sufficiently small.
Furthermore, in the field of ensemble optimal control, an existence result in the $L^2$-setting using ODE similar to \eqref{eq:ODE-Net1} is proved by Scagliotti in \cite[Theorem 3.2]{Scagliotti23}. 
In \cite{Pogodaev2016}, Pogodaev proved the existence of optimal control of the continuity equation with parameters $\theta$ relaxed to Young measures on a bounded domain.
As a corollary, the existence of optimal parameters follows if the neural network satisfies certain convexity conditions, but the continuity of an optimal curve $\mu^\ast$ with respect to $t$ is not clear.

\cref{thm:existence_of_kineticODE_Net} also provides one theoretical justification for the experimental algorithm in \cite{pmlr-v119-finlay20a}.
The authors developed an algorithm to approximate the minimizer of the Benamou--Brenier type problem, which is guaranteed to exist.
However, the guarantee does not hold for the algorithm because the vector field $v$ in the continuity equation is constrained by the neural network $v_\theta$.
This study supplements the existence of a minimizer, even in this case.

In \cref{sec:ideal}, we combine the neural network property of universal approximation with the training of ODE-Net in \cref{prob:ideal_prob}. This new combination makes it possible to obtain existence results (\cref{thm:ideal_prob}) without the linearity assumption (\cref{assump:linear_NN}).
It is also interesting that \cref{prob:ideal_prob} has a similar formulation to (variational) Mean-Field Game (MFG)~\cite{Lasry2007,Benamou2017,Santambrogio2020}.
This similarity between Deep Learning and MFG has recently been pointed out by E et al.~\cite{E2018} and Ruthotto et al.~\cite{ruthotto2020machine}. 
Our results are expected to suggest a strong connection between MFG and ODE-Net.
In fact, for the proof of \cref{thm:ideal_prob}, we will give an auxiliary theorem (\cref{thm:super_existence}) that is proved via the so-called Lagrange perspective for easy handling of the vector fields $v$ (see also \cite[Subsection 2.2.2]{Santambrogio2020}).

\subsection{Organization of the paper}
This paper is organized as follows.
In \cref{subsec:background}, we will give a brief review of previous studies on ODE-Net. In the first half, we summarize the history of the development of ODE-Net, and in the second half, we review mathematical formulations of the learning of ODE-Net.
In \cref{sec:prelim}, we will provide preliminary facts on the convergence of probability measures and distributional solutions of the continuity equation, which will be used to set up and prove our main results.
In \cref{sec:kin}, we will prove \cref{thm:existence_of_kineticODE_Net}. By virtue of the regularization term in \eqref{eq:J_kin} and the Benamou--Brenier formula in \cref{lem:BenamouBrenier}, we will easily get the appropriate compactness of minimizing sequences.
Hence, we can apply the direct method of the Calculus of Variations to reach the existence results.
In \cref{sec:ideal}, we will exhibit how \cref{prob:ideal_prob} is formulated through an idealization in a detailed manner, and then we prove our main result (\cref{thm:ideal_prob}). One cannot prove the theorem by simply applying the arguments used in \cref{sec:kin}.
Instead, we show the theorem by the use of a supplementary problem (see \cref{prob:super_ideal_prob,thm:super_existence}) based on the Lagrange perspective.
\cref{sec:conclusion} presents a summary of the paper and discusses some tasks to be undertaken in future studies.
In \cref{appendix:H^1,appendix:convexity}, we show and review the existence results for problems 
given by Bonnet et al.~\cite{bonnet2022measure} and Thorpe and Gennip~\cite{thorpe2020deep}. These problems adopt different regularization terms from \cref{prob:kinetic_prob}. 
By comparing the proofs of \cref{thm:existence_of_kineticODE_Net,thm:existence_of_H^1,thm:existence_conv}, one can observe differences in how compactness is obtained to minimizing sequences.

\section{Background and related works}\label{subsec:background}
This section provides an overview of previous research on learning of ODE-Net. \cref{subsec:origin} reviews how ODE-Net has been proposed. \cref{subsec:formulations} describes how the learning has been formulated and discussed.

\subsection{Background to the development of ODE-Net}\label{subsec:origin}
Before describing the history of the development of ODE-Net, we shall review a type of DNN called ResNet that led to the improvement of DNN's performance.
ResNet was devised to facilitate optimization of DNN in \cite{resnet}. The simplest $L$-layer ResNet consists of the difference equation
\begin{equation}
    \begin{aligned}
    x_{0}&=g(x,\theta),\\
    x_{t+1}&=x_t+v(x_t,\theta_t),\quad t=0,\dots,L-1,\\
    y      &=h(x_L,\theta_L),
\end{aligned}
\label{eq:difference_eq}
\end{equation}
where $x\in\R^d$ is an input data, $y\in\Y$ denotes a final output, and $g(\bullet,\theta)\colon\R^d\to\R^{d_0}$ and $h(\bullet,\theta_L)\colon\R^{d_L}\to\Y\subset\R^{d_{\Y}}$ are some linear maps with parameters $\theta\in\R^{d_0\times d}$ and $\theta_L\in\R^{d_{\Y}\times d_L}$ respectively. In addition, $v(\bullet,\theta_t)\colon\R^{d_t}\to\R^{d_{t+1}}$ is multiple compositions of some affine maps with $\theta_t$, and nonlinear functions, called activation functions, such as Rectified Linear Unit (ReLU)~\cite{Nair10ReLU}.
Out of various models of (Deep) Neural Networks, we shall refer to the above mapping $v(\bullet, \theta_t)$ associated with ResNet as a neural network simply in this paper.
Experimentally, ResNet is known to perform better than other DNNs. In particular, \emph{deep} ResNet, i.e., \eqref{eq:difference_eq} with $L\gg1$ outperforms other Machine Learning methods.

When ResNet is very deep, it is natural to observe ResNet~\eqref{eq:difference_eq} as the explicit Euler discretization of an ODE with unit step size. 
With the pioneering works in \cite{Haber17,E2017,sonoda2017double}, a trend started to analyze DNNs and develop algorithms by replacing ``discrete'' DNNs with ``continuum'' ODEs.
 For example, Haber et al.~\cite{Haber17} employed the linear stability analysis in the theory of dynamical systems to stabilize ResNet, and Lorin et al.~\cite{Lorin2020} utilized the parallel computing for differential equations to speed up the training of ResNet.
These ``continuum'' ODEs corresponding to DNNs are often called Neural ODE in \cite{ChenRBD18}, or ODE-Net in \cite{Zhong2020Symplectic,NEURIPS2021_b5d62aa6}. 
Specifically, the following parameterized dynamical system is often called \emph{ODE-Net}:
\begin{equation}
    \begin{aligned}
    x_{0}&=g(x,\theta),\\
    \dot{x}_t&=v(x_t,\theta_t),\quad t\in(0,T),\\
    y      &=h(x_T,\theta_T),
\end{aligned}
\label{eq:ode}
\end{equation}
where $x\in\R^d$, $y\in\Y$, $g\colon\R^d\times\R^{d\times d}\to\R^d$, $h\colon\R^d\times\R^{d_{\Y}\times d}\to\Y$ and $v$ are defined as in \eqref{eq:difference_eq}. Note that for simplicity, it is assumed that $x_t\in \R^d$ for any $t\in[0,T]$, and accordingly, the neural network $v(\bullet,\theta_t)$ becomes a vector field on $\R^d$. Also, the finite-dimensional parameters $\theta_0,\theta_1,\dots,$ and $\theta_{L-1}$ in \eqref{eq:difference_eq} are replaced with a (measurable) function on $[0,T]$. While $\theta\colon[0,T]\to\R^m$ is sometimes supposed to be continuous for theoretical reasons, the function $\theta$ on $[0,T]$ can be discontinuous during the learning process as seen in \cite[Figure 2]{dissecting} and \cite[Figure 1]{baravdish2022learningCG}. Thus, we impose the Lebesgue integrability condition on $\theta$ in our setting. The terminal time $T>0$ is an arbitrary given constant.

Although there is not so much mathematical research on ODE-Net, the basic properties of general DNNs have also been studied for ODE-Net.
For example, ODE-Net has universal approximation properties 
 proved by \cite{teshima2020universal} and that the objective functional $J$ has no local minima shown in \cite{lu20b,ding22overparam,ding21on}. It is also known that specific additional assumptions (e.g., continuity of $\theta\colon[0,T]\to\R^m$) are necessary to regard ResNet as the discretization of ODE-Net (see, e.g., \cite{sander2022do,jabir2021meanfieldneural,thorpe2020deep}) and to guarantee the convergence of learning algorithms in \cite{jabir2021meanfieldneural}. 
 
 \subsection{Formulations of the learning of ODE-Net and existence results}\label{subsec:formulations}
Practically, people want ODE-Net to output a desired $y$ for an input $x$.
For this purpose, \emph{ODE-Net needs to learn}, i.e., we optimize the parameter $\theta$ in ODE-Net \eqref{eq:ode}.
Thus, it is necessary to establish a theory of the learning of ODE-Net.
E et al.\ were the first to attempt a general formulation of the learning of ODE-Net \eqref{eq:ode} in \cite{E2017,E2018}.
They modeled the learning as a mean-field optimal control problem as follows:

\begin{problem}[Learning problem constrained by ODE-Net~{\cite[Equation 3]{E2018}}]\label{prob:ODE_Net_E}
Let $\Y=\R^l$, let $\Theta$ be a subset of $\R^m$ and let $v\colon\R^d\times\R^m\to\R^d$, $\ell\colon\R^d\times\Y\to\R_+$ and $L\colon\R^d\times\R^m\to\R_+$ be continuous. 
For a given input data $\mu_0\in\mathcal{P}_c(\R^d\times\Y)$, the learning problem constrained by ODE-Net is posed as the following constrained minimization problem:
\begin{align}
    &\minimize_{\subalign{&\theta\in L^\infty\qty(0,T;\Theta)}} {\mathbb{E}\qty[\ell(x_T,y)+\int_0^T L(x_t,\theta_t)\dd{t}]},
    \label{eq:objective1}\\
    &\text{subject to }
    \left\{
    \begin{aligned}
    &\dot{x}_t=v(x_t,\theta_t), t\in(0,T),\\
    &(x_0,y)\sim\mu_{0}.
    \end{aligned}
    \label{eq:ODE-Net1}
    \right.
\end{align}
\end{problem}
\noindent
The meanings of symbols appearing in \cref{prob:ODE_Net_E} are as follows. The given probability measure $\mu_0$ is called training data; a probability distribution of input-output pairs of a random variable $(x,y)$ in \eqref{eq:ode} used for the learning.
The vector field $v(\bullet,\theta)$, $\theta\in\R^m$, on $\R^d$ represents the neural network explained in \eqref{eq:ode}.
After expanded by the linearity of the expected values, the first term of \eqref{eq:objective1} represents the expected value of a loss function $\ell(x,y)$, which is the target we want to make as small as possible during the learning process. One often uses the squared loss $\ell(x,y)=\abs{x-y}^2/2$ for regression problems or the cross-entropy for classification problems (see, e.g., \href{https://pytorch.org/docs/stable/generated/torch.nn.CrossEntropyLoss.html}{Pytorch's document} for the specific form). 
However, when using a neural network with many parameters, minimizing only the loss $\mathbb{E}\qty[\ell(x_T,y)]$ can lead to the so-called overfitting; see basic statistics and machine learning textbooks, e.g., \cite[Subsection 1.4.7]{murphy2013machine}.
To avoid this overfitting, we also minimize the second expected value, which is called a regularization term. For example, some researchers use the $L^2$-regularization $L(x,\theta)=\lambda\abs{\theta}^2/2$, $L^1$-regularization $L(x,\theta)=\lambda\abs{\theta}$, and entropy regularization used in \cite{jabir2021meanfieldneural,pmlr-v70-haarnoja17a}. In addition, the kinetic regularization $L(x,\theta)=\lambda\abs{v(x,\theta)}^2/2$ that Finlay et al.\ proposed with the help of the Benamou--Brenier formula in \cite{pmlr-v119-finlay20a} can make the trajectories of ODEs' solutions well-behaved.
Another way to deal with the overfitting is to restrict $\Theta\subset\R^m$ to compact sets. 
In Optimal Control Theory, by virtue of the compactness of $\Theta$, one can easily show the existence of optimal parameters (see, e.g., \cite[Theorem 5.1.1]{bressan2007}).
It should be noted that these various regularizations require an assumption upon a function space to which the parameters $\theta$ belong. 
As is seen in the above \cref{prob:ODE_Net_E}, E et al.\ set the function space to $L^\infty$-space in \cite{E2018}.

\begin{remark}[On neglecting input and output transformations in~\eqref{eq:ode}]\label{rmk:in_out}
ODE-Net introduced in \eqref{eq:ode} contains input and output transformations $g$ and $h$, leading to a learning problem corresponding to a minimization with respect to $\theta\in\R^{d\times d}$, $\theta_\bullet\in L^2(0,T;\R^m)$ and $\theta_L\in\R^{d_\Y\times d}$. However, current theoretical studies of ODE-Net often use formulations that ignore $g(x,\theta)$ and $h(x,\theta_L)$, and consider minimization only in $\theta_t$ as in \cref{prob:ODE_Net_E}.
In the author's view, the reason for this neglect is that the existence of minimizers for $\theta$ and $\theta_L$ is easy to check if one proposes a variational formulation that considers $g$ and $h$.
For example, if one imposes the $L^2$-regularization $\abs{\theta}^2+\abs{\theta_L}^2$ for the parameters $\theta\in\R^{d\times d}$ and $\theta_L\in\R^{d_\Y\times d}$ associated with $g(\bullet,\theta)\colon x\mapsto x_0$ and $h(\bullet,\theta_L)\colon x_T\mapsto y$ respectively, the existence of minimizers $\theta^\ast\in\R^{d\times d}$ and $\theta_L^\ast\in\R^{d_\Y\times d}$ follows immediately by virtue of the direct method of the Calculus of Variations; a minimizing sequence of $((\theta^n, \theta_L^n))_n$ has a convergent subsequence thanks to the Bolzano--Weierstrass theorem.
Hence, only the ODE $\dot{x}_t=v(x_t,\theta_t)$ in \eqref{eq:ode} is sometimes referred to as ODE-Net.
    On the other hand, $g$ and $h$ should \emph{not} be ignored when we explore the learning process, that is, the dynamics of solving the problem with mathematical optimization methods such as GD. It is reported that singular values of a parameter defining $g$ and $h$ affect the convergence of GD~\cite[Theorem 2]{barboni2021global}.
\end{remark}

On the other hand, for \cref{prob:ODE_Net_E}, Bonnet et al.\ brought a measure-theoretical formulation inspired by mean-field optimal control problems~\cite[Section 1.4]{bonnet2022measure}. 
A trick used in their formulation is that laws $\mu_t$, $t\in(0,T)$, of random variables $(x_t,y)$ subject to \eqref{eq:ODE-Net1} satisfy the following continuity equation:
\begin{equation*}
\left\{
\begin{aligned}
        \partial_t\mu_t+\Div_x(\mu_t(x,y)v(x,\theta_t))&=0,\ \qty(x,y)\in \R^d\times\Y,\ t\in (0,T),\\
    \left.\mu_t\right|_{t = 0}&=\mu_0,       
\end{aligned}
\right.
\end{equation*}
 in the sense of distributions defined in \cref{def:wk_sol}. They utilized this trick to translate \cref{prob:ODE_Net_E} into the following \cref{prob:ODE_Net} in the case of $L(x,\theta)=\lambda\abs{\theta}^2$:
 
\begin{problem}[Measure-theoretical learning problem~{\cite[Equation 1.8]{bonnet2022measure}}]\label{prob:ODE_Net}
Let $\lambda > 0$ be constants and let $v\colon\R^d\times\R^m\to\R^d$ and $\ell\colon\R^d\times\Y\to\R_+$ be continuous. 
For a given input data $\mu_0\in\mathcal{P}_c(\R^d\times\R^d)$, the learning problem constrained by ODE-Net is posed as the following constrained minimization problem: 
\begin{align}
    &\minimize_{\subalign{&\theta\in L^2\qty(0,T;\R^m)}} \qty{\int_{\mathbb{R}^d\times\R^d}\ell(x,y)\dd{\mu_T(x,y)}+{\lambda}\int_0^T\abs{\theta_t}^2\dd{t}},
    \label{eq:objective}\\
    &\text{subject to }
    \left\{
    \begin{aligned}
    \partial_{t} \mu_{t}+\Div_x\left(v\left(x, \theta_{t}\right) \mu_{t}\right)&=0, \quad (x,y)\in\R^d\times\R^d, \ t\in(0,T), \\ 
    \left.\mu_{t}\right|_{t=0}&=\mu_{0}.
    \end{aligned}
    \label{eq:ODE-Net}
    \right.
\end{align}
In addition, $\mu$ belongs to $C_w([0,T];\mathcal{P}_{c}(\R^d\times\R^d))$ which is the space of narrowly continuous curves (see also \cref{def:measure}).
\end{problem}
\noindent
As for \cref{prob:ODE_Net}, Bonnet et al.\ studied the unique existence of a minimizer $\theta^\ast$ under the assumption that $\lambda>0$ is \emph{sufficiently large} and the neural network $v(x,\theta)$ is bounded for $\theta$ 
~\cite[Theorem 3.2]{bonnet2021properties}.
 In practice, however, in order to minimize the loss, the regularization parameter $\lambda$ is usually set to be a \emph{sufficiently small} positive number rather than a large one.

The difficulty in obtaining existence theorems to \cref{prob:ODE_Net} is attributed to the variational formulation. From \eqref{eq:objective} and \eqref{eq:ODE-Net}, we observe that the learning of ODE-Net has the following aspects:
\begin{enumerate}[label=(\roman*),ref=(\roman*)]
    \item the objective functional $J$ in \eqref{eq:objective} is minimized over an infinite-dimensional space $L^2(0,T;\R^m)$, and\label{enum:infinite}
    \item the minimization is constrained by the continuity equation~\eqref{eq:ODE-Net} which is a differential equation on the infinite-dimensional space of probability measures $\mathcal{P}(\R^d\times\Y)$.\label{enum:constrained}
\end{enumerate}
When one tries to show the existence of a minimizer for a variational problem such as \cref{prob:ODE_Net} by using the direct method of the Calculus of Variations, it is difficult to obtain the strong compactness of minimizing sequences due to the infinite dimensionality in \ref{enum:infinite}.
In addition, even if minimizing sequences converge, it is not generally obvious whether limits satisfy the continuity equation~\eqref{eq:ODE-Net} mentioned in \ref{enum:constrained}.

\section{Preliminaries}\label{sec:prelim}
This section presents fundamental mathematical tools.

\subsection{Compactness lemma}\label{sec:meas}
For $T>0$, we denote by $C([0,T];X)$ the set of continuous mappings from $[0,T]$ to a topological space $X$ with the uniform convergence topology.
\begin{lemma}[Ascoli--Arzel\'{a}'s theorem]\label{lem:ascoli}
Let $(X,d)$ be a metric space. Then, a family $\mathcal{F}\subset C([0,T];X)$ is relatively compact in the uniform convergence topology if and only if
\begin{itemize}
    \item for each $t\in [0,T]$, the set $\Set{x\in X|x=f(t)\text{ for some }f\in\mathcal{F}}$ is relatively compact in $X$, and 
    \item $\mathcal{F}$ is equi-continuous.
\end{itemize}
\end{lemma}
\begin{proof}
A more general version of the above lemma in the case where $X$ is a uniform space is proved in, e.g., \cite[Chapter 7.17]{Kelley175}
\end{proof}

\subsection{Probability measures and the Wasserstein space}\label{subsec:Wass}
Hereinafter, $\mathcal{P}(X)$ denotes the set of Borel probability measures on a separable metric space $X$.
Here, we review some definitions and lemmas regarding properties and convergence of probability measures, as well as properties of the Wasserstein space.

\begin{definition}\label{def:measure}
Let $p\ge 1$ and let $(X,d)$ be a Polish space, i.e., a complete and separable metric space.
\begin{enumerate}
\item (narrow convergence)\;
A sequence $(\mu^n)$ in $\mathcal{P}(X)$ is said to be narrowly convergent to $\mu\in\mathcal{P}(X)$ as $n\to\infty$ if 
\[
\lim_{n\to\infty}\int_X f\dd{\mu^n}=\int_X f\dd{\mu}\text{ for every function $f\in C_{{b}}(X)$,}
\]
where $C_{{b}}(X)$ is the space of continuous and bounded real functions defined on $X$. A topology induced by the convergence is said to be the narrow topology.

\item (uniformly integrable $p$-moments)\;
A subset $K$ in $\mathcal{P}(X)$ has \emph{uniformly integrable $p$-moments} if 
\[
\adjustlimits\lim_{R\to\infty}\sup_{\mu\in K}\int_{X\setminus B_X(R,\overline{x})}d(x,\overline{x})^p\dd{\mu(x)}=0 \quad\mbox{for some}\; \overline{x}\in X,
\]
where $B_X(R,x)$ is the open ball of radius $R$ and center $x$ in $X$.
\item (finite $p$-th moment)\; 
A probability measure $\mu\in\mathcal{P}(X)$ is said to have the finite $p$-th moment if 
\begin{align*}
\int_X d\qty(x,\overline{x})^p\dd{\mu(x)} <\infty \quad\mbox{for some}\; \overline{x}\in X,
\end{align*}
and the set of probability measures on $X$ with the finite $p$-th moment is denoted by $\mathcal{P}_p(X)$.

\item (function of $p$-growth)\; 
A function $f\colon X\to \R$ is said to have $p$-growth if there exist $A,B\ge 0$ and $\overline{x}\in X$ such that $|f(x)|\le A+B(d(x,\overline{x}))^p$ for all $x\in X$.

\item (Wasserstein distance)\;
The ($L^p$-)Wasserstein distance between $\mu^1,\mu^2\in\mathcal{P}_p(X)$ is defined by 
\begin{align*}
W_p(\mu^1,\mu^2)&\coloneqq\inf\Set{\left(\int_{X^2}d(x_1,x_2)^p\dd{\pi(x_1,x_2)}\right)^{1/p}|\pi\in\Gamma(\mu^1,\mu^2)},
\end{align*}
where $\Gamma(\mu^1,\mu^2)$ denotes the set of all Borel probability measures $\pi$ on $X^2$ such that for any measurable subset $A\subset X$,
\[
\pi\qty[A\times X]=\mu^1\qty[A],\quad\pi\qty[X\times A]=\mu^2\qty[A].
\]
\end{enumerate}
\end{definition}


By using the H\"{o}lder inequality, one easily gets 
\begin{corollary}\label{cor:WassHolder}
Let $1\le p<q <\infty$, let $X$ be a Polish space and let $\mu^1,\mu^2 \in \mathcal{P}_q(X)$. Then $W_p(\mu^1,\mu^2) \le W_q(\mu^1,\mu^2)$.
\end{corollary}

\begin{lemma}[Kantrovich--Rubinstein duality~{\cite[Thoerem 1.14]{VillaniTopic}}]\label{lem:Kantrovich}
Let $(X,d)$ be a Polish space and let $\rho_0,\rho_1\in\mathcal{P}_{1}(X)$. Then
    \[W_1(\rho_0,\rho_1)=\sup\Set{\int_{X}\varphi\dd{(\rho_1-\rho_0)}|\varphi\in L^1(\abs{\rho_1-\rho_0}),\ \Lip_{X}(\varphi)\coloneqq\sup_{x\neq y\in X}\frac{\abs{\varphi(x)-\varphi(y)}}{d(x,y)}\leq 1}.\]
\end{lemma}
\begin{proof}
    See \cite[Section 11.8]{dudley_2002}.
\end{proof}

A sufficient condition for a family with the uniformly integrable $p$-moments is known, and the proof of the following lemma is given for the sake of the reader's convenience.

\begin{lemma}[{\cite[Subsection 5.1.1]{AGS}}]\label{lem:uni_mom}
Let $p\geq1$. If a subset $K\subset\mathcal{P}(X)$ satisfies
\[
\sup_{\mu\in K}\int_X {d(x,\overline{x})}^{p_1}\dd{\mu(x)}<+\infty,
\]
for some $p_1>p$ and $\overline{x}\in X$, then $K$ has uniformly integrable $p$-moments.
\end{lemma}

The following lemma shows a fine criterion that reveals whether a sequence $(\mu^n)\subset\mathcal{P}(X)$ has the uniformly integrable $p$-moments.

\begin{lemma}[Narrow convergence for $p$-growth functions]\label{lem:narrow_p}
A sequence $(\mu^n)$ in $\mathcal{P}(X)$ has uniformly integrable $p$-moments if and only if 
\begin{enumerate}
\item the sequence is narrowly convergent to $\mu\in\mathcal{P}(X)$, and 
\item for every continuous function $f\colon X\to\R$ of $p$-growth, 
\[
    \lim_{n\to\infty}\int_X f\dd{\mu^n}=\int_X f\dd{\mu} 
\]
\end{enumerate}
\end{lemma}
\begin{proof} 
See \cite[Lemma 5.1.7]{AGS}.
\end{proof}
By \cref{lem:narrow_p} and \cite[Proposition 7.1.5]{AGS}, convergence in  $W_p$ and narrow convergence for $p$-growth functions are equivalent.
\subsection{Continuity equation}
The following definition and lemma are based on a famous text \cite[Chapter 4]{AGS}, to which we refer the reader who wants a general discussion of the continuity equations.
\begin{definition}[Solutions in the sense of distributions]\label{def:wk_sol}
Let $T>0$. 
A continuous curve $\mu\in C_w(\qty[0,T];\mathcal{P}(\R^d\times\Y))$ is called a solution to the continuity equation
\begin{equation}\label{eq:cont_eq}
\partial_t\mu_t+\Div_x\qty(v_t\mu_t) =0 \text{ in }(0,T)\times\R^d\times\Y,    
\end{equation}
in the sense of distribution, if
\begin{equation}
\int_{0}^{T} \int_{\mathbb{R}^{d}\times\Y}\left(\partial_{t} \psi_t(x,y)+\nabla_{x} \psi_t(x,y) \cdot v_t(x)\right) \mathrm{d} \mu_{t}(x,y) \mathrm{d} t=0\label{eq:eqweak}    
\end{equation}
for every $\psi\in\Ccinf((0,T)\times\R^d\times\Y)$.
Here a mapping $v_t \colon\R^d\ni x\mapsto v_t(x)\in \R^d$, $t\in [0,T]$, is a Borel vector field.
\end{definition}
In the following, we adopt \cref{def:wk_sol} as the solution of the continuity equation~\eqref{eq:cont_eq} with a vector field $v$.
\begin{lemma}[Representation formula for \eqref{eq:cont_eq} {\cite[Proposition 8.1.8]{AGS}}]\label{lem:ODErep}
Let $T>0$ and let $\mu\in C_w([0,T];\mathcal{P}(\R^d\times\Y))$ be a distributional solution of \eqref{eq:cont_eq} with Borel vector fields $v=(v_t)_t$. Assume that $v$ satisfies that
\begin{align}
    \int_0^T\intRdY\abs{v_t}\dd{\mu_t}\dd{t}<\infty,\label{eq:integrable1}
\end{align}
and
\begin{equation}\label{eq:Lipschitz}
    \int_0^T\qty(\sup_K\abs{v_t}+\Lip_{K}(v_t))\dd{t}<\infty
    \quad\mbox{for every compact set}\ K\subset\R^d.
\end{equation}
Here $\Lip_K(v_t)$ denotes a Lipschitz constant of the mapping $v_t\colon\R^d\to\R^d$ on $K$, i.e., \[\Lip_K(v_t)\coloneqq\sup_{x\neq y\in K}\frac{\abs{v_t(y)-v_t(x)}}{\abs{y-x}}.\]
Then, for $\mu_0$-a.e.\ $(x,y)\in\R^d\times\Y$, there exists a unique solution $X_{\bullet}(x)\in C([0,T];\R^d)$  such that 
\begin{align*}
    X_0(x)&=x,\\
    \dv{t}X_t\qty(x)&=v_t\qty(X_t\qty(x)).
\end{align*}
Furthermore, the solution $\mu_t$ is represented as 
\begin{equation}\label{eq:ODErep}
    \mu_t=\qty(X_t\times\Id_\Y)_{\#}\mu_0 \quad\text{for all}\ t\in [0,T],
\end{equation}
where $\Id_X\colon X\to X$ is the identity mapping on $X$.
\end{lemma}

\begin{proof} 
The existence result can be shown by the use of the standard argument of the Picard iteration method. For the representation result, details are proved in, e.g., \cite[Proposition 8.1.8]{AGS}.
\end{proof}

The following lemma indicates the strong relation between the Wasserstein distance and the continuity equation.

\begin{lemma}[Benamou--Brenier formula~\cite{Benamou2000}]\label{lem:BenamouBrenier}
Let $\rho_0,\rho_1\in\mathcal{P}_{2}(\R^d)$. Then, 
\begin{equation}
    W_2(\rho_0,\rho_1)^2=\inf\Set{\int_0^1\intRd\abs{v_t(x)}^2\dd{\rho_t(x)\dd t}|(\rho,v)\in V(\rho_0,\rho_1)},\label{eq:BBformula}
\end{equation}
where
\begin{align*}
V(\rho_0,\rho_1)\coloneqq\Set{(\rho,v)\in C([0,1];\mathcal{P}_2(\R^d))\times L^2(\dd{\rho_t\dd t}) |
\begin{array}{l}
\eqref{eq:cont_eq}\text{ holds in the sense of}\\
\text{\cref{def:wk_sol}, and}\\
\eval{\rho_t}_{t=0}=\rho_0,\ \eval{\rho_t}_{t=1}=\rho_1.
\end{array}
}.
\end{align*}
\end{lemma}
\begin{proof}
See \cite[Theorem 17.2]{Ambrosio2021}.
\end{proof}

\section{Kinetic Regularization and an Existence Theorem}\label{sec:kin}
In this section, we discuss the existence of a minimizer to the \emph{kinetic regularized} learning problem introduced in \cref{prob:kinetic_prob}. The section begins with some background on kinetic regularization. Subsequently, We proceed to the proof of \cref{thm:existence_of_kineticODE_Net}.
Throughout the paper, we denote by $C$ a generic non-negative constant which may vary from line to line.

\subsection{Kinetic regularization}\label{sec:kin_assump}
In general, to argue minimizers of a functional $J$ as in \eqref{eq:J_kin} via the direct method of the Calculus of Variations, a minimizing sequence $((\mu^n, \theta^n))_n$ of the functional needs to be compact in an appropriate topological space. Moreover, the topology must be sufficiently strong to lead to a (lower semi)continuity of the functional. Driven by this necessity, some previous studies have tried to strengthen the topology of the space of the parameter $\theta\colon[0,T]\to\R^m$ in \cite{thorpe2020deep,bonnet2022measure,herty2022large}. However, this strong topology leads to unusual assumptions, as reviewed in \cref{subsec:formulations}.
Instead, we seek for compactness of the continuous curves $\mu^n \colon[0,T]\to\Pcal(\R^d\times\Y)$, rather than of the parameters $\theta^n$ ($n\in\mathbb{N}$). This idea is rarely seen in Machine Learning but often in the MFG theory (see, e.g.,~\cite[Theorem 6.6.]{ORRIERI20191868} and \cite[Theorem 6]{BONNET2021594}). To illustrate this idea, we need the following lemma derived from the Benamou--Brenier formula (\cref{lem:BenamouBrenier}):

\begin{lemma}[Uniform continuity estimate]\label{lem:speed}
Let $\mu\in C_w([0,T];\Pcal(\R^d\times\Y))$ be a distributional solution to the continuity equation~\eqref{eq:cont_eq} with Borel vector fields $v_t \colon\R^d\ni x\mapsto v_t(x)\in\R^d$, $t\in[0,T]$. Then it holds that 
\begin{equation*}
    W_2\qty(\mu_t,\mu_s)^2\leq(s-t)\int_t^s\intRdY\abs{v(\tau,x)}^2\dd{\mu_\tau(x,y)\dd\tau}
\end{equation*}
for $0\leq t<s\leq T$.
\end{lemma}
\noindent
In the rest of the paper, we often abbreviate $\intRdY f(x)\dd{\mu(x,y)}$ to $\intRd f\dd{\mu}$ for a function $f\colon\R^d\to\R$ independent of $y\in\Y$.
\begin{proof} 
From \cref{lem:BenamouBrenier}, we have
\begin{align*}
    &W_2\qty(\mu_t,\mu_s)^2\\
    \leq&\inf_{\rho,w}\Set{\int_0^1\intRdY\abs{w_t(x)}^2\dd{\rho_t(x,y)\dd t}|\partial_t\rho+\Div_x\qty(w\rho)=0,\rho_0=\mu_t,\rho_1=\mu_s.}\\
    =&\inf_{\rho,w}\Set{\int_t^s\intRd\abs{w_\tau}^2\dd{\rho_\tau}\dd{\frac{\tau}{s-t}}|\qty(s-t)\partial_\tau\rho+\Div_x\qty(w\rho)=0,\rho_t=\mu_t,\rho_s=\mu_s.}\\
    =&(s-t)\inf\Set{\int_t^s\intRd\abs{w_\tau}^2\dd{\rho_\tau\dd \tau}|\partial_\tau\rho+\Div_x\qty(w\rho)=0,\rho_t=\mu_t,\rho_s=\mu_s.}\\
    \leq& (s-t)\int_t^s\intRd\abs{v_\tau}^2\dd{\mu_\tau}\dd{\tau}
\end{align*}
for $0\leq t<s\leq T$. 
\end{proof}

This lemma readily leads to the following:
\begin{corollary}\label{cor:equiconti}
Let $n\in\mathbb{N}$ and let $\mu^n\in C_w([0,T];\mathcal{P}_c(\R^d\times\mathcal{Y}))$ be a distributional solution to the continuity equation~\eqref{eq:cont_eq} corresponding to Borel vector fields $(v_t^n)_{t\in[0,T]}$. 
If 
\begin{align*}
\sup_{n\in\mathbb{N}}\int_0^1\int_{\R^d} |v_t^n|^2\, \dd\mu_t\dd t <\infty,
\end{align*}
then the family $(\mu^n)$ is equi-continuous.
\end{corollary}

To use \cref{cor:equiconti} explicitly, we add a term
\begin{equation}
    \frac{\lambda}{2}\abs{v\qty(x,\theta)}^2,\quad\lambda>0,
\end{equation}
 to the objective functional $J$ \eqref{eq:J_kin} in \cref{prob:kinetic_prob}.
This regularization term $\abs{v\qty(x,\theta)}^{2}$ is reported to be effective in generative models in \cite{pmlr-v119-finlay20a}.
Here, we use kinetic regularization for simplicity, but in fact, one can prove the existence of a minimizer without a kinetic regularization term. See \cref{rmk:other_cases} for details.
\subsection{Existence theorem}\label{sec:existence_theorem}
Our strategy is 
to use the direct method of the Calculus of Variations, containing the following three steps:
\begin{enumerate}
    \item take a minimizing sequence $((\mu^n,\theta^n))_n$ and extract a convergent subsequence in suitable topologies,
    \item check that $J$ is lower semicontinuous with respect to those topologies and 
    \item verify that the limits of convergent subsequences satisfy the constraint~\eqref{eq:ODE2}.
\end{enumerate}
As for a minimizing sequence, we get a weakly convergent subsequence of $(\theta^n)$ in $L^2(0,T;\R^m)$ and the strongly convergent subsequence of $(\mu^n)$ in $C([0,T];(\mathcal{P}_2(\R^d\times\mathcal{Y}),W_2))$ by virtue of~\cref{cor:equiconti} and the Ascoli--Arzel\'{a} theorem. 
From these convergences and \cref{assump:linear_NN}, we observe that the functional $J$ is lower semicontinuous in $(\mu,\theta)$.
Also, we can verify that the limits solve the continuity equation again.
This is why we impose the kinetic regularization term onto the functional $J$. 

For the proof of \cref{thm:existence_of_kineticODE_Net}, we need a lemma on the boundedness of the support of $\mu_t$ uniformly in $t \in [0,T]$. 

\begin{lemma}\label{lem:cptspt}
Let $\theta\in L^2(0,T;\R^m)$ and let $\mu\in C([0,T];(\mathcal{P}_2(\R^d\times\Y),W_2))$ be a distributional solution of \eqref{eq:ODE2} corresponding to vector fields $(v(\bullet,\theta_t))_{t\in[0,T]}$. If \cref{assump:linear_NN} holds, there exists a radius $R^\ast=R^\ast\qty({\mu_0,f,T,\norm{\theta}_{L^2\qty(0,T;\R^m)}})>0$ such that
\[\supp\mu_t\subset B_{\R^d\times\Y}(R^\ast,0)\eqqcolon B_{\R^d\times\Y}(R^\ast)\text{ for all $t\in[0,T]$.}\]
\end{lemma}

\begin{proof} 
To use \cref{lem:ODErep}, we check the assumption~\eqref{eq:integrable1} and \eqref{eq:Lipschitz}. By \cref{assump:linear_NN} and the Lipschitz continuity of $f$, we have
\begin{align*}
    \int_0^T\int_{\R^d\times\Y}\abs{v(x,\theta_t)}\dd{\mu_t}\dd{t}&\leq\int_0^T\int_{\R^d\times\Y}\abs{\theta_t}\abs{f(x)} \dd{\mu_t}\dd{t}\\
    &\leq\norm{\theta}_{L^2(0,T;\R^m)}\sqrt{\int_0^T\qty(\intRdY\abs{f(x)}\dd{\mu_t(x,y)})^2\dd{t}}\\
    &\leq\sqrt{T}\norm{\theta}_{L^2(0,T;\R^m)}\sqrt{\sup_{t\in[0,T]}\intRdY\abs{f(x)}^2\dd{\mu_t(x,y)}}\\
    &\leq C\norm{\theta}_{L^2(0,T;\R^m)}\qty(1+\sup_{t\in[0,T]}\sqrt{\int_{\R^d\times\Y}\abs{x}^2\dd{\mu_t(x,y)}})<\infty 
\end{align*}
Similarly, it holds that
\begin{align*}
    \int_0^T\qty(\sup_K\abs{v(\bullet,\theta_t)}+\Lip_K\abs{v(\bullet,\theta_t)})\dd{t}\leq C(T+\norm{\theta}_{L^2(0,T;\R^m)}^2)<\infty,
\end{align*}
for every compact set $K\subset\R^d\times\Y$.
We thus find from~\cref{lem:ODErep} that $\mu_t$ can be represented as $\mu_t=\qty(X_t,\Id_\Y)_\#\mu_0$ where $X_t\colon\R^d\to\R^d$ is the flow maps of the corresponding ODE satisfying
\begin{align*}
 &\left\{ \,
\begin{aligned}
    \dv{X_t(x)}{t}&=v(X_t(x),\theta_t)\ \text{ for }t\in(0,T),\\
    X_0(x)&=x,
\end{aligned}
\right.
 \end{align*}
for almost all $(x,y)\in\supp\mu_0$. By Gr\"{o}nwall's inequality and \cref{assump:linear_NN}, we have
    \begin{align*}
        \abs{X_t(x)}&\leq\qty(\abs{x}+C\int_0^T\abs{\theta_s}\dd{s})\exp(C\int_0^T\abs{\theta_s}\dd{s})\\
        &\leq C \qty(\diam\qty(\supp\mu_0)+T+\sqrt{T}\norm{\theta}_{L^2(0,T;\R^m)})\exp(C\sqrt{T}\norm{\theta}_{L^2(0,T;\R^m)}) \\ 
        &\le R^\ast
    \end{align*}
for some $R^\ast>0$ independent of $t$ since $\theta\in L^2(0,T;\R^m)$,
whence follows $\supp \mu_t \subset B(R^\ast)$ for all $t\in [0,T]$.
\end{proof}

\begin{remark}
    \cref{assump:linear_NN} can be generalised to \cref{assump:condNN} when one only proves \cref{lem:cptspt}. See also \cref{lem:cptspt_H^1}.
\end{remark}

With \cref{lem:cptspt}, one can now proceed with the proof of \cref{thm:existence_of_kineticODE_Net}.
\begin{proof}[Proof of~{\cref{thm:existence_of_kineticODE_Net}}]
Set
\[
S=\Set{(\mu,\theta)\in C\qty([0,T];(\mathcal{P}_2(\R^d\times\Y),W_2))\times L^2\qty(0,T;\R^m)|(\mu,\theta)\text{ satisfies \eqref{eq:ODE2}}}.
\]
Since $(\mu_0,0)\in S$ is a trivial and regular solution to the continuity equation, we see that $S\neq\emptyset$.
It is also obvious that $0\le \inf J <+\infty$ because the integrand $\ell$ is non-negative. Then, there exists a minimizing sequence $\qty(\qty(\mu^n,\theta^n))_{n=1}^\infty\subset S$ such that $J\qty(\mu^n,\theta^n)\to\inf_S J$ as $n\to\infty$. For the sequence, there exists a constant $C>0$ independent of $n$ such that
\begin{align}
\frac{\lambda}{2}\int_0^T\intRd\abs{\theta_t^n f\qty(x)}^2\dd{\mu_t^n\qty(x)\dd t}\leq C,    \label{eq:kinertic_bound}\\
\frac{\epsilon}{2}\int_0^T\abs{\theta^n_t}^2\dd{t}\leq C.\label{eq:parameter_bound}
\end{align}
From \cref{lem:speed} and \eqref{eq:kinertic_bound}, we have for $0\le t<s\le T$, 
\begin{align*}
    W_2\qty(\mu_t^n,\mu_s^n)^2\leq(s-t)\int_t^s\intRd\abs{\theta_\tau^n f\qty(x)}^2\dd{\mu_\tau^n\qty(x)}\dd{\tau}\leq\frac{2C}{\lambda}(s-t). 
\end{align*}
Hence it follows that $\qty(\mu^n)\subset C([0,T];(\mathcal{P}_2(\R^d\times\Y),W_2))$ is equi-continuous.
Also, by \cref{lem:cptspt} and \eqref{eq:parameter_bound}, there exists a constant $R^\ast>0$ independent of $n$ and $t$ such that
\[
\mu_t^n\in\Set{\mu\in\Pc(\R^d\times\Y)|\supp\mu\subset B_{\R^d\times\Y}(R^\ast)}\quad\text{ for all }n\in\mathbb{N}\text{ and all }t\in[0,T].
\] 
In addition, the set $\Set{\mu\in\Pc(\R^d\times\Y)|\supp\mu\subset B_{\R^d\times\Y}(R^\ast)}$ is compact with respect to $L^2$-Wasserstein topology because of \cite[Proposition 7.1.5]{AGS}.
Hence, \cref{lem:ascoli} and\eqref{eq:parameter_bound} imply that there exist a subsequence of $(n)$, still denoted by $n$,
$\mu^\ast\in C([0,T];(\mathcal{P}_2(\R^d\times\Y),W_2))$ and $\theta^\ast\in L^2(0,T;\R^m)$ such that
\begin{align}
\mu^{n}\to \mu^\ast &\text{ strongly in }C \qty(\qty[0,T];\qty(\mathcal{P}_2(\R^d\times\Y),W_2)),\label{eq:mu_uniform} \\ 
\theta^{n}\to\theta^\ast &\text{ weakly in }L^2\qty(0,T;\R^m).\label{eq:theta_weak}  
\end{align}
By the following \cref{claim:CE}, we can deduce that $(\mu^\ast,\theta^\ast)$ solves \eqref{eq:ODE2} in the sense of distribution.

\begin{claim}\label{claim:CE}
For the limits $\mu^\ast$ and $\theta^\ast$, it holds that
\begin{equation}
\int_{0}^{T} \int_{\R^{d}\times\Y}\left(\partial_{t} \zeta_t+\nabla_{x} \zeta_t \cdot v\qty(\bullet,\theta^\ast_t)\right) \dd{\mu_t^{\ast}\dd t}=0,   \label{eq:CEast} 
\end{equation}
for all $\zeta\in \Ccinf\qty(\qty(0,T)\times\R^d\times\Y)$.
Moreover, $\supp \mu^\ast \subset B_{\R^d\times\Y}(R)$ for some $R>0$.
\end{claim}

\begin{proof} 
We already know that
\begin{equation}
\int_{0}^{T} \int_{\R^{d}\times\Y}\partial_{t} \zeta_t\dd{\mu_t^{n}\dd t}+\int_{0}^{T} \int_{\R^{d}\times\Y}\nabla_{x} \zeta_t \cdot v(x,\theta^n_t) \dd{\mu_t^{n}\dd t}=0,\label{eq:CEn2}
\end{equation}
for all $\zeta\in \Ccinf\qty(\qty(0,T)\times\R^d\times\Y)$ and $n\in\mathbb{N}$. 
It follows that 
\begin{align}
\begin{split}
0 &= 
\int_{0}^{T} \int_{\R^{d}\times\Y}\partial_{t} \zeta_t\, \dd(\mu_t^{n}-\mu^\ast_t)\dd t \\ 
&\quad +\int_{0}^{T} \int_{\R^{d}\times\Y}\nabla_{x} \zeta_t(x) \cdot v(x,\theta^{\ast}_t) \dd{\qty(\mu_t^{n}-\mu^\ast_t)(x)}\dd {t}
\\ 
&\quad\quad +\int_{0}^{T} \int_{\R^{d}\times\Y}\nabla_{x} \zeta_t(x) \cdot \qty(v\qty(x,\theta^{n}_t)-v\qty(x,\theta^\ast_t))\, \dd\mu^\ast_t(x)\dd t \\
&\quad\quad\quad  +\int_{0}^{T} \int_{\R^{d}\times\Y}\nabla_{x} \zeta_t(x) \cdot \qty(v\qty(x,\theta^{n}_t)-v\qty(x,\theta^\ast_t))\, \dd(\mu_t^{n}(x)-\mu^\ast_t(x))\dd t \\
&\quad\quad\quad\quad  + 
\int_0^T\int_{\R^d\times\Y} \partial_t\zeta_t\, \dd\mu^\ast_t\dd t +
\int_{0}^{T} \int_{\R^{d}\times\Y}\nabla_{x} \zeta_t(x) \cdot v(x,\theta^{\ast}_t) \dd{\mu_t^{\ast}(x)\dd t}
\\
&\eqqcolon I_1 + I_2 + I_3 + I_4 + I_5.
\end{split}
\label{eq:claim}
\end{align}
It follow from~\eqref{eq:mu_uniform} that 
\begin{align*}
\abs{I_1}
&\le 
\int_0^T \Lip_{\R^d\times\Y}(\partial_t\zeta_t) \int_{\R^d\times\Y} \frac{\partial_t\zeta_t}{\Lip_{\R^d\times\Y}(\partial_t\zeta_t)}  \, \dd(\mu^n_t-\mu^\ast_t)\dd t \\ 
&\leq C \int_0^T W_1(\mu^n_t,\mu^\ast_t)\dd{t}
\\ 
&\leq C\int_0^TW_2(\mu^n_t,\mu^\ast_t)\dd{t}
\\
&\leq CT\sup_{t\in[0,T]}W_2(\mu^n_t,\mu^\ast_t)\to 0, 
\end{align*}
as $n\to\infty$ by the Kantrovich--Rubinstein duality (\cref{lem:Kantrovich}) and \cref{cor:WassHolder}. 
By~\cref{assump:linear_NN}, the function $\partial_t\zeta_t(x) v(x,\theta^\ast_t)$ is Lipschitz continuous in $x$ and $y$ over $\R^d\times\Y$, and thus we see again from \cref{lem:Kantrovich,cor:WassHolder} that 
\begin{align*}
\abs{I_2}
&\leq C \int_0^T W_1(\mu^n_t,\mu^\ast_t)\dd{t}
\\ 
&\leq C\int_0^TW_2(\mu^n_t,\mu^\ast_t)\dd{t}
\\
&\leq CT\sup_{t\in[0,T]}W_2(\mu^n_t,\mu^\ast_t)\to 0, 
\end{align*}
as $n\to\infty$. 

For $I_3$, we use \cref{assump:linear_NN} to apply \eqref{eq:theta_weak}. We set 
\[
\varphi_{\bullet}\coloneqq\intRdY\nabla_x\zeta_\bullet f^\top\dd{\mu_\bullet^\ast}\in L^2\qty(0,T;\R^m).
\]
In fact, it is shown that  
\[
\norm{\varphi}_{L^2\qty(0,T;\R^m)}^2\leq\int_0^T\intRdY\abs{\nabla_x\zeta_t f^\top}^2\dd{\mu_t^n\dd t}\leq C\qty(1+\sup_{t\in[0,T]}W_2(\mu_t^n,\delta_0))<\infty.
\]
Then, we can deduce that
\[
    I_3=\int_0^T\la\varphi_t,\theta^n_t-\theta^\ast_t\ra\dd{t}\to0\text{ as }n\to\infty,
\]
where $\la A,B\ra\coloneqq\Tr(A^\top B)$, $A,B\in\R^{d\times p}$ is the inner product on $\R^{d\times p}$.

As for $I_4$, it follows from~\eqref{eq:mu_uniform} and~\eqref{eq:theta_weak} that 
\begin{align*}
|I_4| 
&=
\left| \int_0^T \left\langle \theta^n_t-\theta^\ast_t, \int_{\R^d\times\Y} \nabla_x\zeta_{\bullet} f^T\, \dd(\mu^n_t-\mu^\ast_t) \right\rangle \,\dd t  
\right| \\
&\le C
\int_0^T \abs{\theta^n_t-\theta^\ast_t} W_1(\mu^n_t,\mu^\ast_t) \,\dd t \\ 
&\le C  \norm{\theta^n-\theta^\ast}_{L^2(0,T;\R^m)}
\sup_{t\in [0,T]} W_2(\mu^n_t,\mu^\ast_t)
\\ 
&\le C  \sup_{t\in [0,T]} W_2(\mu^n_t,\mu^\ast_t)\to 0.
\end{align*}
Passing to the limit as $n\to\infty$ in~\eqref{eq:claim}, we get~\eqref{eq:CEast}.

By the lower semicontinuity of the $L^2$-norm, we have 
$\norm{\theta^\ast_{\bullet}}_{L^2(0,T;\R^m)}\le \sup_{n\in\mathbb{N}}\norm{\theta^n_{\bullet}}_{L^2(0,T;\R^m)}<\infty.$
Then \cref{lem:cptspt} implies that there exists $R>0$ such that 
$\supp \mu^\ast_t \subset B_{\R^d\times\Y}(R)$ for all $t\in [0,T]$, whence follows the conclusion.
\end{proof}
We resume the proof of \cref{thm:existence_of_kineticODE_Net}.
From \cref{claim:CE}, we have $(\mu^\ast,\theta^\ast)\in S$. We then show that $J\qty(\mu^n,\theta^n)\to J(\mu^\ast,\theta^\ast)$ as $n\to\infty$. First, from \eqref{eq:mu_uniform}, \cref{assump:label_loss}, \cref{lem:narrow_p} and \cite[Proposition 7.1.5]{AGS}, it follows that 
\begin{equation}
    \intRdY\ell\dd{\mu^n_T}\to\intRdY\ell\dd{\mu^\ast_T},\label{eq:continuity_ell}
\end{equation}
as $n\to\infty$. We next estimate the regularization term. Again from \eqref{eq:parameter_bound} and \eqref{eq:mu_uniform}, we infer that
\begin{equation}
    \begin{aligned}
    &\abs{\int_0^T\intRd\abs{\theta^n_t f\qty(x)}^2\dd{\mu^n_t\qty(x)\dd t}-\int_0^T\intRd\abs{\theta^n_t f\qty(x)}^2\dd{\mu^\ast_t\qty(x)\dd t}}\\
    =&\abs{\int_0^T\la\theta^n_t\intRd ff^\top\dd{\qty(\mu^n_t-\mu^\ast_t)},\theta^n_t\ra\dd{t}}\\
    \leq&\int_0^T\abs{\intRd ff^\top\dd{\qty(\mu^n_t-\mu^\ast_t)}}\abs{\theta_t^n}^2\dd{t}
    \leq\frac{2C}{\epsilon}\max_{t\in[0,T]}\abs{\intRd ff^\top\dd{\qty(\mu^n_t-\mu^\ast_t)}}
    \to0\text{ as }n\to\infty.
\end{aligned}
    \label{eq:kin_lsc}
\end{equation}
Now we set 
\[
    \normiii*{\theta}_{L^2\qty(0,T;\R^m)}^2\coloneqq\int_0^T\la\theta_t\qty(\lambda\intRd ff^\top\dd{\mu^\ast_t}+\epsilon),\theta_t\ra\dd{t},
\]
for $\theta\in L^2\qty(0,T;\R^m)$. Since $\displaystyle\lambda\intRd ff^\top\dd{\mu^\ast_t}+\epsilon$ is positive definite matrix, the function $\normiii*{\bullet}_{L^2\qty(0,T;\R^m)}$ defines an equivalent norm of $L^2\qty(0,T;\R^m)$. Hence, it follows from \cite[Proposition 3.5]{Brezis2011} that $\normiii*{\bullet}_{L^2\qty(0,T;\R^m)}$ is weakly lower semicontinuous. Thus, we conclude that
\begin{equation}
    \begin{split}
    \inf_S J 
    &=\liminf_{n\to\infty} J\qty(\mu^n,\theta^n)\\
    &=\intRdY\ell\dd{\mu^\ast_T}+\liminf_{n\to\infty}\int_0^T\intRd\qty(\frac{\lambda}{2}\abs{\theta_t^n f\qty(x)}^2+\frac{\epsilon}{2}\abs{\theta^n_t}^2)\dd{\mu^n_t\qty(x)\dd t}\\
    &=\intRdY\ell\dd{\mu^\ast_T}+\liminf_{n\to\infty}\int_0^T\intRd\qty(\frac{\lambda}{2}\Tr\qty(\theta_t^nf(x)f(x)^\top{\theta_t^n}^\top)+\frac{\epsilon}{2}\abs{\theta^n_t}^2)\dd{\mu^n_t\qty(x)\dd t}\\
    &=\intRdY\ell\dd{\mu^\ast_T}+\liminf_{n\to\infty}\int_0^T\intRd\qty(\frac{\lambda}{2}\la\theta_t^nf(x)f(x)^\top,\theta_t^n\ra+\frac{\epsilon}{2}\la\theta_t^n,\theta_t^n\ra)\dd{\mu^n_t\qty(x)\dd t}\\
    &=\intRdY\ell\dd{\mu^\ast_T}+\frac{1}{2}\liminf_{n\to\infty}\int_0^T\la\theta_t^n\qty(\lambda\intRd ff^\top\dd{\mu^n_t}+\epsilon),\theta_t^n\ra\dd{t}\\
    &\geq\intRdY\ell\dd{\mu^\ast_T}+\frac{1}{2}\liminf_{n\to\infty}\normiii*{\theta^n}^2_{L^2\qty(0,T;\R^m)}\\
    &+\frac{\lambda}{2}\liminf_{n\to\infty}\qty(\int_0^T\intRd\abs{\theta^n_t f\qty(x)}^2\dd{\mu^n_t\qty(x)\dd t}-\int_0^T\intRd\abs{\theta^n_t f\qty(x)}^2\dd{\mu^\ast_t\qty(x)\dd t})\\
    &\geq\intRdY\ell\dd{\mu^\ast_T}+\frac{1}{2}\normiii*{\theta^\ast}^2_{L^2\qty(0,T;\R^m)}+0\\
    &=J\qty(\mu^\ast,\theta^\ast)
    \geq\inf_S J,
\end{split}
\label{eq:lsc}
\end{equation}
i.e., $J(\mu^\ast,\theta^\ast)=\inf_S J$, and the proof of \cref{thm:existence_of_kineticODE_Net} is complete.
\end{proof}

\begin{remark}[The case of $\lambda=0$]\label{rmk:other_cases}
If one only wants to show the existence of a minimizer, it is sufficient to use only $\epsilon\abs{\theta}^2/2$ as the regularization term. In other words, we can prove the theorem when $\lambda=0$. Indeed, by \cref{lem:cptspt}  it is apparent that
\[
    \intRd\abs{x}^2\dd{\mu_t^n}\leq (R^\ast)^2
\]
holds for every $t\in[0,T]$ and $n\in\mathbb{N}$, where $R^\ast>0$ is the same as the radius in \cref{lem:cptspt}. Thus, we obtain 
\begin{align*}
    W_2\qty(\mu_t^n,\mu_s^n)^2&\leq(s-t)\int_t^s\intRd\abs{\theta_\tau^n f\qty(x)}^2\dd{\mu_\tau^n\qty(x)}\dd{\tau}\\
    &\leq(s-t)\qty(\norm{\theta}^2_{L^2\qty(s,t;\R^m)}\norm{\intRd\abs{f(x)}^2\dd{\mu_\bullet^n\qty(x)}}_{L^\infty\qty(s,t)})\\
    &\leq\frac{2C(s-t)}{\epsilon}\norm{\intRd\abs{f(x)}^2\dd{\mu_\bullet^n\qty(x)}}_{L^\infty\qty(s,t)}\\
    &\leq\frac{2C(s-t)}{\epsilon}\norm{\qty(\Lip f)^2\intRd\abs{x}^2\dd{\mu_\bullet^n\qty(x)}+\abs{f(0)}^2}_{L^\infty\qty(s,t)}\\
    &\leq\frac{2C\qty(\qty(R^\ast\Lip f)^2+\abs{f(0)}^2)(s-t)}{\epsilon}
\end{align*}
from \cref{lem:speed,assump:linear_NN}, or \eqref{eq:vNorm}, and \eqref{eq:parameter_bound}. Here $\Lip f\geq0$ is a Lipschitz constant of $f$. Consequently, we can guarantee the equi-continuity of the curve $\mu$ without the kinetic regularization term. Then, we complete the proof using an argument similar to the one above. In this case, the proof of the lower semicontinuity \eqref{eq:lsc} becomes rather simple. 
It is noteworthy, however, that even in this case, deriving the convergence of $I_3$ in the proof of \cref{claim:CE} from the weak convergence of $\theta$ is difficult without imposing \cref{assump:linear_NN}. In this sense, it seems essential under $L^2$-regularization that the neural network is linear with respect to the parameters. If continuity of $\theta$ can be obtained, e.g., by $H^1$-regularization, then the convergence of $I_3$ can be easily shown. See the proof of \cref{thm:existence_of_H^1}.
\end{remark}

\begin{remark}[Uniqueness of a minimizer]\label{rmk:unique}
When a regularization parameter $\epsilon$ is sufficiently large, the uniqueness of the minimizer is proved by Bonnet et al.~\cite[Theorem 3.2]{bonnet2022measure}. 
For self-containdness, we will present details in \cref{appendix:convexity}. On the contrary, we do not refer to the uniqueness in \cref{thm:existence_of_kineticODE_Net} since $\epsilon$ might be small in most practical cases.
\end{remark}
\section{Ideal Learning Problem}\label{sec:ideal}
This section discusses the existence of a minimizer to the \emph{ideal} learning problem as introduced in \cref{prob:ideal_prob}. At the beginning of this section, we explain why we consider this problem before we prove the main theorem (\cref{thm:ideal_prob}).

\subsection{Idealization of learning problems}
Neural networks $v$ in \eqref{eq:ODE2} have, in general, a complex structure, while \cref{assump:linear_NN} imposes the simplicity of linearity of $v$ in $\theta$ to prove \cref{thm:existence_of_kineticODE_Net}.
However, it is difficult to show \cref{claim:CE} under the general assumption because the trick described in \eqref{eq:lsc} is unavailable.

Thus, we assume that $v(\bullet,\theta)$ can be any square-integrable vector field. This assumption might be justified by
\emph{universal approximation properties} resulting from the complexity. Universal approximation means that the set of functions expressed by a neural network $\Set{v(\bullet, \theta)\colon\R^d\to\R^d | \theta\in\R^m}$ is dense in appropriate function spaces (for example, Lebesgue spaces $L^p(\R^d)$). We refer the reader to \cite{Cybenko1989,HORNIK1991251,pinkus_1999} for details.
In light of these results, one can infer that 
\[
\overline{\Set{\R^d\times[0,T]\ni(x,t)\longmapsto v(x,\theta_t)\in\R^d|\theta\in L^2(0,T;\R^m)}}^{\norm{\bullet}_{L^2(\dd\mu_t\dd t)}}= L^2\qty(\dd{\mu_t\dd t}),
\]
holds where, by abuse of notation, we set for a fixed $\mu\in C([0, T];\mathcal{P}(\R^d))$
\begin{equation}
    \begin{aligned}
            &L^2\qty(\dd{\mu_t\dd t})\\
            \coloneqq&\Set{(v_t)_{t\in[0,T]}\text{ is a family of Borel vector fields on $\R^d$}|\norm{v}^2_{L^2(\dd\mu_t\dd t)}\coloneqq\int_0^T\int_{\R^d}\abs{v_t}^2\dd{\mu_t\dd t}<+\infty}.\label{eq:L^2dmudt}
    \end{aligned}
\end{equation}
This abuse is referred to in the notation in Villani's text~\cite[Equation (8.6)]{VillaniTopic}. Furthermore, if $\epsilon=0$ in \eqref{eq:J_kin}, $\theta$ only appears via $v$ in \cref{prob:kinetic_prob}. 

Therefore, we can regard \cref{prob:kinetic_prob} as a problem about vector fields (neural network) $v$, rather than parameters $\theta$. 
From the above, we consider \cref{prob:ideal_prob} as the further idealized learning problem.
The problem is similar to the variational form of MFG introduced by Lasry and Lions~\cite{LASRY2006679,Lasry2007}. We also refer the reader to more comprehensive lecture notes by Santambrosio~\cite[Subsection 2.2]{Santambrogio2020}. 
\subsection{Proof of the existence via the Lagrangian framework}\label{sec:super_ideal}
We then consider the existence of a minimizer in \cref{prob:ideal_prob}. For this problem, we want to apply a similar argument to \cref{thm:existence_of_kineticODE_Net}. However, unlike the previous problem, the space $L^2(\dd\mu_t\dd t)$ depends on $\mu$, rendering it intractable.
In such cases, it is helpful to rewrite the problem with ``probability measure on curves'' as the variable instead of ``curve on probability measures'' or $\mu\in C([0,T];\mathcal{P}(\R^d))$ as the variable. That is, we consider $Q\in\mathcal{P}(\R^d\times C([0,T];\R^d))$ as presented in \cref{prop:prob_rep}. Here $AC^2([0,T];\R^d)$ denotes the set of an absolutely continuous curve $\gamma\colon[0,T]\to\R^d$ such that there exists $m\in L^2(0,T)$ satisfying $\abs{\gamma(s)-\gamma(t)}\leq\int_{s}^t m(\tau)\dd{\tau}$ for all $s,t\in[0,T]$, $s<t$.

\begin{proposition}[Probabilistic representation]\label{prop:prob_rep}
Let $\mu\in C([0,T];(\mathcal{P}_2(\R^d),W_2))$ satisfy the continuity equation $\partial_t\mu_t+\Div\qty(v_t\mu_t)=0$ in the distributional sense for a Borel vector field $v_t$ such that
\begin{equation}
    \int_0^T\intRd\abs{v_t}^2\dd{\mu_t\dd t}<+\infty. \label{eq:integrable}
\end{equation}
Then, there exists $Q\in \mathcal{P}(\R^d\times C(\qty[0,T];\R^d))$ such that
\begin{enumerate}[label=(\roman*),ref=(\roman*)]
    \item\label{enum:concentrate} $Q$ is concentrated on the set of pairs $(x,\gamma)$ such that $\gamma\in AC^2(\qty[0,T],\R^d)$ is an absolutely continuous solution of $\dot{\gamma}(t)=v_t(\gamma(t))$ for a.a.\ $t\in(0,T)$ with $\gamma(0)=x$; 
    \item $\mu_t=\mu^{Q}_t$ for any $t\in\qty[0,T]$, where $\mu_t^{Q}$ is defined as 
    \begin{equation}
        \intRd\varphi\dd{\mu_t^{Q}}\coloneqq\int\limits_{\R^d\times C(\qty[0,T];\R^d)}\varphi\qty(\gamma\qty(t))\dd{Q(x,\gamma)},\label{eq:rewrite}        
    \end{equation}
    for all $\varphi\in C_b(\R^d)$.
\end{enumerate}
Conversely, if $Q\in \mathcal{P}(\R^d\times C(\qty[0,T];\R^s))$ satisfies \ref{enum:concentrate} and 
\begin{equation}
    \int\limits_{\R^d\times C([0,T];\R^d)}\int\limits_{0}^T\abs{\dot{\gamma}(t)}^2\dd{t\dd Q(x,\gamma)}<+\infty,\label{eq:integQ}
\end{equation}
then there exists $\mu^Q\in C([0,T];(\mathcal{P}_2(\R^d),W_2))$ induced via \eqref{eq:rewrite}, which is a solution of the continuity equation with the following vector field
    \[
        \tilde{{v}}_t(x)\coloneqq\int\limits_{\Set{\gamma \in C([0, T] ; \R^d)| \gamma(t)=x}} \dot{\gamma}(t) \dd{Q_x(\gamma)}\in L^2(\dd{\mu^Q_t\dd t}) \quad \text { for } \mu_t^Q \text {-a.e. } x \in \R^d,
    \]
where $Q_x$ is the disintegrated measures with respect to the evaluation map $\mathsf{e}_t\colon\R^d\times C([0,T];\R^d)\ni(x,\gamma)\longmapsto\gamma(t)\in\R^d$.
\end{proposition}

\begin{proof} 
See proofs of \cite[Theorem 8.2.1]{AGS}, \cite[Theorem 5.3]{Ambrosio2004}, and \cite[Theorems 4 and 5]{Lisini2007}.
\end{proof}
Recently, in studies for MFG~\cite{Benamou2017,Santambrogio2020}, considering $\mu$ instead of $Q$ is called the Lagrange perspective. 
This perspective is summed up in the slogan ``Think Eulerian, prove Lagrangian'' in \cite[Chapter 15]{villani_oldnew}, which is widely applied in, e.g., \cite{sarrazin22}.
We refer the reader to \cite[Subsection 8.2]{AGS} and \cite{Lisini2007} for a more general theory. 

Referring to the formulation in MFG, we rewrite the ideal \cref{prob:ideal_prob} in terms of $Q$.
Before starting a more ideal problem, we introduce an evaluation map
\begin{align*}
    \mathsf{e}_{t}&\colon(y, \gamma) \in \mathcal{Y}\times C(\qty[0,T];\R^d) \longmapsto \qty(\gamma(t),y) \in \mathbb{R}^{d}\times\mathcal{Y}, \quad \text { for } t \in[0, T].
\end{align*}
With the above, one can rewrite \cref{prob:ideal_prob} as follows: 

\begin{problem}[Ideal learning problem in the Lagrangian framework]\label{prob:super_ideal_prob}
Let $\lambda > 0$ be a constant, let $\ell\colon\R^d\times\Y\to\R_+$ be continuous, and  let $\mu_0\in\Pc(\R^d\times\Y)$ be a given input data.
Set 
\begin{align}
\widetilde{J}(Q)
\coloneqq
\int_{\mathcal{Y}\times C(\qty[0,T];\R^d)}\qty(\ell(\gamma(1),y)
+\int_0^T\frac{\lambda}{2}\abs{\dot\gamma(t)}^2\dd t)\dd Q(y,\gamma),
\label{eq:J_lag}
\end{align}
for $Q\in \mathcal{P}(\mathcal{Y}\times C(\qty[0,T];\R^d))$.
Then, the ideal learning problem in the Lagrangian framework is posed as the following constrained minimization problem:
\begin{align*}
    \inf\Set{\widetilde{J}(Q)|Q\in \mathcal{P}(\mathcal{Y}\times C(\qty[0,T];\R^d))\text{ such that }\qty(\mathsf{e}_0)_{\#}Q=\mu_0}.
\end{align*}
\end{problem}

Comparing \cref{prob:ideal_prob} and \cref{prob:super_ideal_prob}, the functional $\widehat{J}$ in \eqref{eq:J_ideal} and $\widetilde{J}$ in \eqref{eq:J_lag} have the correspondence such that
\begin{align*}
    \int_{\mathbb{R}^d\times\Y}\ell(x,y)\dd{\mu_T(x,y)}
&\Longleftrightarrow \int_{\mathcal{Y}\times C(\qty[0,T];\R^d)}\ell(\gamma(1),y)
\dd{Q(y,\gamma)},\\
    \int_0^T\int_{\mathbb{R}^d\times\Y} {\frac{\lambda}{2}\abs{v\qty(x,t)}^2}\dd{\mu_t(x,y)\dd t}&\Longleftrightarrow \int_{\mathcal{Y}\times C(\qty[0,T];\R^d)}\int_0^T\frac{\lambda}{2}\abs{\dot\gamma(t)}^2\dd{t\dd Q(y,\gamma)}.
\end{align*}
We see that \cref{prob:super_ideal_prob} has fewer constraints and fewer variables than \cref{prob:ideal_prob}.
This is because, according to \cref{prob:ideal_prob}, $\mu$ and $v$ satisfying the continuity equation \eqref{eq:ODE2} can be recovered as long as $Q$ is obtained.
This fact leads us to the existence of a minimizer for \cref{prob:super_ideal_prob}.
\begin{lemma}[Existence result for \cref{prob:super_ideal_prob}]\label{thm:super_existence}
Under \cref{assump:label_loss}, there exists a minimizer $Q\in \mathcal{P}(\mathcal{Y}\times C(\qty[0,T];\R^d))$ for \cref{prob:super_ideal_prob}.
\end{lemma}

\begin{proof} 
Set 
\[
    S=\Set{Q\in \Pcal(\Y\times C(\qty[0,T];\R^d))|\qty(\mathsf{e}_0)_{\#}Q=\mu_0},
\]
here the probability measures $\mathcal{P}(\mathcal{Y}\times C(\qty[0,T];\R^d))$ are endowed with the narrowly convergence topology.
We can easily check that a measure $\qty(\mathsf{e}_{\mathcal{Y}}\times\mathsf{c}_{\bullet})_\#\mu_0$ belongs to $S$, where we set $
\operatorname{\mathsf e}_\Y\colon\R^d\times\mathcal{Y}\ni\qty(x,y)\longmapsto y\in\mathcal{Y}$ and $\mathsf{c}\colon\R^d\times\Y\ni\qty(x,y)\longmapsto\qty{\qty[0,T]\ni t\longmapsto x\in\R^d}\in C\qty([0,T];\R^d)$. It is clear that $0\leq\widetilde{J}<+\infty$ since the integrand $\ell$ is non-negative. Thus, we take a minimizing sequence $\qty(Q^n)_{n}\subset S$ such that $\widetilde{J}\qty(Q^n)\to\inf J\in\R$ as $n\to\infty$. From the second term of \eqref{eq:J_lag}, there exists a constant $C$ independent of $n$ such that
\begin{equation}
    \frac{\lambda}{2} \int\limits_{\Y\times C\qty(\qty[0,T];\R^d)}\int\limits_0^T\abs{\dot{\gamma}(t)}^2\dd{t\dd Q^n(y,\gamma)}\leq C.
    \label{eq:kinetic_bound2}
\end{equation}
\par
Next, we claim that $(Q^n)_n$ is tight. We choose the maps $r^1$ and $r^2$ defined on $\Y\times C\qty([0,T];\R^d)$ as
\[
r^1\colon(y,\gamma)\longmapsto y\in\Y,\quad r^2\colon(y,\gamma)\longmapsto\gamma\in C\qty([0,T];\R^d).
\]
It is clear that $\qty(r^1_\#Q^n)_n$ is tight because of \cref{assump:label_loss} and Prokhorov's theorem.
In addition, the functional 
\begin{equation}
A\colon C(\qty[0,T];\R^d)\ni\gamma\longmapsto
\left\{
    \begin{aligned}
        &\int_0^T\frac{
        \lambda}{2}\abs{\dot{\gamma}(t)}^2\dd{t} &&\qty(
        \text{
        \begin{tabular}{l}
             if $\gamma$ is an absolutely continuous curve\\
             with $\abs{\dot{\gamma}}\in L^2(0,T)$ and $\gamma(0)\in\supp\mu_0$.
        \end{tabular}
        }
        )\\
        &+\infty  &&\qty(\text{otherwise})
    \end{aligned}
    \right.   \label{eq:defA}  
\end{equation}
has a compact sublevel sets in $C\qty(\qty[0,T];\R^d)$ because of the Ascoli--Arzel\'{a} theorem. Hence we can see that $\qty(r^2_\#Q^n)_n$ is also tight thanks to an integral condition for tightness~\cite[Remark 5.1.5]{AGS} and \eqref{eq:kinetic_bound2}. Then, we obtain the tightness of $\qty(Q^n)_n$ by applying a tightness criterion~\cite[Lemma 5.2.2]{AGS} for the maps $r^1$ and $r^2$.\par
Therefore, there exists a subsequence $\qty(n)$, still denoted by $n$, and $Q^\ast\in\mathcal{P}(\Y\times C([0,T];\R^d))$ such that
\[
Q^n\rightharpoonup Q^\ast\text{ in }\mathcal{P}(\Y\times C([0,T];\R^d)),
\]
by Prokhorov's theorem.\par
It remains to be verified that the limit $Q^\ast$ satisfies $\qty(\mathsf{e}_0)_\#Q=\mu_0$ and $\widetilde{J}\qty(Q^\ast)=\inf\widetilde{J}(=\lim_{n\to\infty}\widetilde{J}\qty(Q^n))$. The former is obtained by the continuity of the evaluation map $\mathsf{e}_t,\ t\in[0,T]$. The latter is shown as follows. By the continuity of $\ell$ and $\mathsf{e}_T$, we obtain that
\begin{align*}
\lim_{n\to\infty}\int\limits_{\R^d\times\Y}\ell\dd{\qty(\mathsf{e}_T)_{\#}Q^n}&=\lim_{n\to\infty}\int\limits_{\R^d\times\Y}\min\qty{\ell\circ\mathsf{e}_T, C'}\dd{Q^n}\\
&=\int\limits_{\R^d\times\Y}\min\qty{\ell\circ\mathsf{e}_T, C'}\dd{Q^\ast}\\
&=\int\limits_{\R^d\times\Y}\ell\dd{\qty(\mathsf{e}_T)_{\#}Q^\ast}    
\end{align*}
for a sufficiently large constant $C'>0$. In addition, the functional $A$ in \eqref{eq:defA} is lower semicontinuous, and we can choose $A^k\in C_{{b}}\qty(C([0,T];\R^d)),k=1,2,\dots,$ such that $A^k\nearrow A$ as $k\to\infty$ by \cite[Theorem 10.2]{Ambrosio2021}. Then, we get that for each $k\in\mathbb{N}$,
\[
\liminf_{n\to\infty}\int\limits_{C\qty([0,T];\R^d)}A\dd{Q^n}\geq\liminf_{n\to\infty}\int\limits_{C\qty([0,T];\R^d)}A^k\dd{Q^n}=\int\limits_{C\qty([0,T];\R^d)}A^k\dd{Q^\ast}.
\]
Hence, passing to the limit as $k\to\infty$ in the above inequality, we obtain 
\[
\liminf_{n\to\infty}\int\limits_{C\qty([0,T];\R^d)}A\dd{Q^n}\geq\lim_{k\to\infty}\int\limits_{C\qty([0,T];\R^d)}A^k\dd{Q^\ast}=\int\limits_{C\qty([0,T];\R^d)}A\dd{Q^\ast},
\]
by virtue of Fatou's lemma.
\end{proof}
From \cref{thm:super_existence,prop:prob_rep} we immediately obtain \cref{thm:ideal_prob}.

\begin{proof}[Proof of \cref{thm:ideal_prob}]
Let $Q^\ast$ denote the minimizer of $\widetilde{J}$. From \cref{prop:prob_rep}, we can get $\mu^{Q^\ast}\in C([0,T];\mathcal{P}_2(\R^d\times\Y))$ satisfying \eqref{eq:integQ} and $\tilde{v}\in L^2(\dd{\mu^{Q^\ast}_t\dd t})$. By \cite[Theorem 5]{Lisini2007}, we have $\widetilde{J}(Q^\ast)=\widehat{J}(\mu^{Q^\ast},\tilde{v})$. From this equality and \cref{prop:prob_rep}, it follows that $\widehat{J}(\mu^{Q^\ast},\tilde{v})\leq \widehat{J}(\mu,v)$ for any $\mu\in C([0,T];\mathcal{P}_2(\R^d\times\Y))$ and $v\in L^2(\dd\mu_t\dd t)$.
\end{proof}

\section{Conclusion}\label{sec:conclusion}
In this paper, we introduced the kinetic regularized learning problem (\cref{prob:kinetic_prob}) and proved the existence of its minimizer in \cref{thm:existence_of_kineticODE_Net}. A key idea in the proof is to show that a sequence of curves $(\mu^n)\subset C([0,T];(\mathcal{P}_2(\R^d\times\Y),W_2))$, rather than a parameter $(\theta^n)\subset L^2(0,T;\R^m)$, converges strongly.
Furthermore, we attempted to idealize \cref{prob:kinetic_prob} as \cref{prob:ideal_prob}, although the relationship between this idealization and the existing neural network is unclear. However, considering the minimizers of \cref{prob:ideal_prob} will provide essential clues for understanding deep learning in the future.

Our results can be further developed through a generalization of neural networks and regularization terms. The directions of each generalization are described below and will be subjects of future work.
\subsection{For general neural network architectures}\label{subsec:general}
It remains to establish an existence result for neural networks more general than \cref{assump:linear_NN}. A general $l$-layer neural network $v$ is a continuous vector field satisfying the following assumptions:
\begin{assumption}[General $l$-layer neural network]\label{assump:condNN}
There exists $C>0$, it holds that
\begin{align}
\abs{v(x,\theta)}&\leq  C\abs{\theta}^l(1+\abs{x}),&&\text{ for }x\in\R^d,\label{eq:vNorm}\\
\abs{v(x_1,\theta)-v(x_2,\theta)}&\leq C\abs{\theta}^l\abs{x_1-x_2},&&\text{ for }x_1,x_2\in\R^d,\label{eq:vLip}
\end{align}
for $\theta\in\R^m$. 
\end{assumption}
\noindent Note that we also assume that $v\colon\R^d\times\R^m\to\R^d$ is continuous in \cref{prob:kinetic_prob}.

For example, $v$ satisfying \cref{assump:linear_NN} is a $1$-layer neural network. In practice, $2,3$-layer neural network is often used. We mentioned in \cref{subsec:aim} that the nonlinearity of such $l$-layer neural networks hinders the proof of existence theorems, especially \cref{claim:CE}.
To relax this nonlinearity, it may be effective to consider a \emph{mean-field} neural network
\begin{equation}
    \mathcal{V}(x,\vartheta)=\int_{\R^m} v(x,\theta)\dd{\vartheta(\theta)},\label{eq:mean_NN}
\end{equation}
where $\vartheta$ is a learnable probability measure on $\R^m$. 
This assumption has long been known as the Young measure~\cite{Castaing2004,Pogodaev2016} in optimal control theory, but it has recently been recognized again as a helpful approach to shallow neural networks\cite{mei18land,NEURIPS2018_a1afc58c,pmlr-v139-akiyama21a} and ODE-Nets~\cite{lu20b,jabir2021meanfieldneural,ding22overparam}. The author is in the process of conducting further theoretical research using this network $\mathcal{V}$.
\appendix
\section{$H^1$-Regularization}\label{appendix:H^1}
We discuss the existence of a minimizer in the same problem setting as \cite{thorpe2020deep}.
\begin{problem}[$H^1$-regularized learning problem]\label{prob:H^1_prob}
Let $\lambda > 0$ be a constant, let $\Y$ be a subset of $\R^d$ and let $v\colon\R^d\times\R^m\to\R^d$ and $\ell\colon\R^d\times\Y\to\R_+$ be continuous. Let $\mu_0\in\mathcal{P}_c(\R^d\times\Y)$ a given input data. Set
\begin{equation}
    J_{H^1}(\mu,\theta)\coloneqq\int_{\mathbb{R}^d\times\Y}\ell\dd{\mu_T}+\frac{\lambda}{2}\norm{\theta}_{H^1(0,T;\R^m)}^2\label{eq:J2}
\end{equation}
for $\mu\in C\qty(\qty[0,T];(\mathcal{P}_2(\R^d\times\mathcal{Y}),W_2))$ and $\theta\in H^1\qty(0,T;\R^m)$.
The $H^1$-regularized learning problem constrained by ODE-Net is posed as the following constrained minimization problem:
\begin{gather} 
    \notag 
    \inf \Set{ J_{H^1}(\mu,\theta) | \mu\in C\qty(\qty[0,T];(\mathcal{P}_2(\R^d\times\mathcal{Y}),W_2)),\; \theta\in H^1\qty(0,T;\R^m)  }, 
    \intertext{subject to } 
    \left\{
    \begin{aligned}
        \partial_{t} \mu_{t}+\Div_{x}\left(v\left(\bullet, \theta_{t}\right) \mu_{t}\right)&=0, \\ 
        \left.\mu_{t}\right|_{t=0}&=\mu_{0},
    \end{aligned}
    \right.\label{eq:ODEH}
\end{gather}
\end{problem}
\noindent We note that the constraint \eqref{eq:ODEH} is the same as \eqref{eq:ODE2}.

For \cref{prob:H^1_prob}, we can obtain an existence result without \cref{assump:linear_NN}.
\begin{theorem}[Existence theorem for \cref{prob:H^1_prob}]\label{thm:existence_of_H^1}
Under \cref{assump:label_loss,assump:condNN}, there exists a minimizer $(\mu,\theta)\in  C\qty(\qty[0,T];(\mathcal{P}_2(\R^d\times\mathcal{Y}),W_2))\times H^1\qty(0,T;\R^m)$ of \eqref{eq:J2} in \cref{prob:H^1_prob}. 
\end{theorem}
Before the proof of \cref{thm:existence_of_H^1}, we prepare a lemma  similar to \cref{lem:cptspt}.
\begin{lemma}\label{lem:cptspt_H^1}
Let $\theta\in H^1(0,T;\R^m)$ and let $\mu\in C([0,T];(\mathcal{P}_2(\R^d\times\Y),W_2))$ be a distributional solution of \eqref{eq:ODEH} corresponding to a vector fields $(v(\bullet,\theta_t))_t$. If \cref{assump:condNN} holds, there exists a radius $R^\ast=R^\ast({\mu_0,f,T,\norm{\theta}_{H^1\qty(0,T;\R^m)}})>0$ such that
\[\supp\mu_t\subset B_{\R^d\times\Y}(R^\ast)\text{ for all $t\in[0,T]$.}\]
\end{lemma}
\begin{proof}
The strategy of the proof is the same as that of the proof of \cref{lem:cptspt}.
 By \cref{assump:condNN} and Sobolev inequality, we have
\begin{align*}
    &\int_0^T\int_{\R^d\times\Y}\abs{v(x,\theta_t)}\dd{\mu_t}\dd{t}\\
    \leq&C\int_0^T\int_{\R^d\times\Y}\abs{\theta_t}^l\qty(1+\abs{x}) \dd{\mu_t}\dd{t}\\
    \leq& C\norm{\theta}^{l}_{L^l(0,T;\R^m)}\qty(1+\sup_{t\in[0,T]}\int_{\R^d\times\Y}\abs{x}\dd{\mu_t})\\
    \leq& C\norm{\theta}^{l}_{H^1(0,T;\R^m)}\qty(1+\sup_{t\in[0,T]}\sqrt{\int_{\R^d\times\Y}\abs{x}^2\dd{\mu_t}})\\
    =& C\norm{\theta}^{l}_{H^1(0,T;\R^m)}\qty(1+\sup_{t\in[0,T]}W_2(\mu_t,\delta))<\infty.
\end{align*}
 Also, again using \cref{assump:condNN} and Sobolev inequality, we can estimate $\abs{X_t(x)}$ in the proof of \cref{lem:cptspt} by $\norm{\theta}_{H^1(0,T;\R^m)}$.
\end{proof}
\begin{proof}[Proof of~{\cref{thm:existence_of_H^1}}]
Set
\[
S=\Set{(\mu,\theta)\in C\qty([0,T];(\mathcal{P}_2(\R^d\times\Y),W_2))\times H^1\qty(0,T;\R^m)|(\mu,\theta)\text{ satisfies \eqref{eq:ODEH}}}.
\]
It is obvious that $S\neq\emptyset$ and $0\leq J_{H^1}<\infty$ on $S$. Then, we can take a minimizing sequence $\qty(\qty(\mu^n,\theta^n))_{n=1}^\infty\subset S$ such that $J_{H^1}\qty(\mu^n,\theta^n)\to\inf_S J_{H^1}$ as $n\to\infty$. By the second term of \eqref{eq:J2}, there exists a constant $C>0$ such that
\begin{equation}
    \frac{\lambda}{2}\norm{\theta^n}^2_{H^1}\leq C,\label{eq:H^1bound}
\end{equation}
for all $n\in\mathbb{N}$.
From \cref{lem:speed}, \cref{assump:condNN}, \eqref{eq:H^1bound} and the Sobolev inequality, we have for $0\leq t<s\leq T$,
\begin{align*}
    W_2\qty(\mu_t^n,\mu_s^n)^2&\leq(s-t)\int_t^s\intRd\abs{v(x,\theta^n_\tau)}^2\dd{\mu_\tau^n\qty(x)\dd{\tau}}\\
    &\leq C(s-t)\int_t^s\intRd\abs{\theta_\tau}^{2l}\qty(1+\abs{x}^2)\dd{\mu_\tau^n\qty(x)\dd{\tau}}\\
    &\leq C(s-t)\int_t^s\intRd\abs{\theta_\tau}^{2l}\qty(1+{R^\ast}^2)\dd{\mu_\tau^n\qty(x)\dd{\tau}}\\
    &\leq C\norm{\theta}_{L^{2l}(0,T;\R^m)}^{2l}(s-t)\\
    &\leq C\norm{\theta}_{H^{1}(0,T;\R^m)}^{2l}(s-t)\\
    &\leq C(s-t),
\end{align*}
where $R^\ast>0$ is the constant appeared in \cref{lem:cptspt_H^1}.
Hence, there exist a subsequence $\qty(n')\coloneqq\qty(n\qty(k))_{k=1}^\infty\subset\Z_{>0}$ and $(\mu^\ast,\theta^\ast)\in C\qty(0,T;(\mathcal{P}(\R^d\times\Y),W_2))\times H^1\qty(0,T;\R^m)$ such that
\begin{align}
\theta^{n'} \to \theta^\ast&\text{ weakly in }H^1\qty(0,T;\R^m), \label{eq:H1wk} \\
\theta^{n'} \to \theta^\ast & \text { strongly in } C(0,T;\R^m), \label{eq:Cstrong}\\
\mu^{n'}\to \mu^\ast&\text{ strongly in }C\qty(0,T;(\mathcal{P}(\R^d\times\Y),W_2)).\label{eq:mu_wk}    
\end{align}
Here, we used the Sobolev embedding theorem in \eqref{eq:Cstrong}.
By the above, we can deduce the following claim:
\begin{claim}\label{claim:CE2}
    The limits $\mu^\ast$ and $\theta^\ast$ satidfy \eqref{eq:CEast} for all $\zeta\in \Ccinf\qty((0,T)\times\R^d\times\Y)$.
\end{claim}
\begin{proof} 
As in the proof of \cref{claim:CE}, the proof is completed by taking the limits of $I_1$ to $I_4$ in \eqref{eq:claim}. From the proof of \cref{claim:CE}, $I_1,I_2\to0$ as $n\to\infty$.
Also, $I_3\to0$ as $n\to\infty$ by the uniform convergence \eqref{eq:Cstrong}, and the continuity of $v(x,\theta)$ with respect to $\theta$.
For $I_4$, it follows from \eqref{eq:vLip}, \eqref{eq:H1wk} and \eqref{eq:mu_wk} that
\begin{align*}
    \abs{I_4}&\leq{\int_0^T\Lip\qty(\nabla_x\zeta_t\cdot\qty(v\qty(\bullet,\theta^n_t)-v\qty(\bullet,\theta_t^\ast)))W_1(\mu^n_t,\mu^\ast_t)\dd{t}}\\
    &\leq C\qty(\norm{\theta^n}_{L^l(0,T;\R^m)}^l+\norm{\theta^\ast}_{L^l(0,T;\R^m)}^l)\sup_{t\in[0,T]}W_1(\mu^n_t,\mu^\ast_t)\\
    &\leq C\qty(\norm{\theta^n}_{H^1(0,T;\R^m)}^l+\norm{\theta^\ast}_{H^1(0,T;\R^m)}^l)\sup_{t\in[0,T]}W_2(\mu^n_t,\mu^\ast_t)\\
    &\leq C\sup_{t\in[0,T]}W_2(\mu^n_t,\mu^\ast_t)\to0\text{ as }n\to\infty.
\end{align*}
Thus we obtain the conclusion.
\end{proof}
We resume the proof of \cref{thm:existence_of_H^1}.
From \cref{claim:CE2}, we have $(\mu^\ast,\theta^\ast)\in S$. In addition, $J_{H^1}$ is lower semicontinuous from \eqref{eq:continuity_ell} and the weak lower semi-continuity of the $H^1$-norm $\norm{\bullet}_{H^1(0,T;\R^m)}$. 
The proof is complete.
\end{proof}
\section{Convexity Assumptions}\label{appendix:convexity}
For comparison, using the proof technique by Bonnet et al.\ \cite{bonnet2022measure}, we show that a unique minimizer to \cref{prob:kinetic_prob} exists. This proof technique makes use of the idea that we can regard the functional $J$ as a univariate functional $\widetilde{J}(\theta)\coloneqq J(\mu^\theta, \theta)$, where $\mu^\theta\in C([0,T];\mathcal{P}(\R^d\times\Y))$ is a solution of \eqref{eq:ODE2} for a given $\theta\in L^2(0,T;\R^m)$. The existence and uniqueness of the solution can be proved by showing the convexity of $\widetilde{J}$. For this purpose, we evaluate the Lipschitz constant of the Fr\'echet derivative $\nabla_\theta\widetilde{J}$ of $\widetilde{J}$.
\par
First, recall from \cref{lem:ODErep} that $\mu^\theta$ is represented as $\mu^\theta_t=\qty(\Phi^\theta(0,t;\bullet)\times\mathrm{Id}_\Y)_\#\mu_0$ using a flow map $\Phi^\theta\colon[0,T]^2\times\R^d\to\R^d$, $\theta\in L^2(0,T;\R^m)$ according to the ODE 
\begin{equation}
    \left\{
    \begin{aligned}
       \partial_t\Phi^\theta\qty(t_0,t;x)&=v\qty(\Phi^\theta(t_0,t;x),\theta_t),\\
      \Phi^\theta(t_0,t_0;x)&=x.
    \end{aligned}
    \right. 
    \label{eq:flow}
\end{equation}
We note here that from \eqref{eq:linear_NN} it can be verified that the Lipschitz continuity assumption \eqref{eq:Lipschitz} is satisfied, as in the proof of \cref{lem:cptspt}.  The derivative of $\Phi$ with respect to $\theta$ can be described by a linearization of \eqref{eq:flow}.
\begin{lemma}[Taylor expansion of $\Phi^\theta$]\label{lem:Taylor}
    Suppose that the neural network $v$ satisfies \cref{assump:linear_NN} and $f\colon\R^d\to\R^p$ is differentiable. Then, for every $\theta,\vartheta\in L^2\qty(0,T;\R^m)$, the Taylor expansion
    \[
    \Phi^{\theta+\epsilon\vartheta}(t_0,t;x)=\Phi^\theta(t_0,t;x)+\epsilon\int_0^t\Delta^\theta_{(s,t)}(x)\vartheta_s f\qty(\Phi^\theta(t,s;x))\dd{s}+o(\epsilon)
    \]
    holds in $C\qty(\qty[0,t]\times\supp\mu_0;\R^d\times\Y)$, where, for $(t_0,x)\in[0,T]\times\R^d$, the map $[0,T]\ni t\longmapsto\Delta^\theta_{(t_0,t)}(\bullet)\in C\qty(\R^d;\R^{d\times d})$ is the unique solution of the linearized Cauchy problem
    \begin{equation}
    \left\{
    \begin{aligned}
       \partial_t\Delta_{(t_0,t)}^\theta\qty(x)&=\theta_t Jf\qty(\Phi^\theta(t_0,t;x))\Delta_{(t_0,t)}^\theta\qty(x),\\
      \Delta_{(t_0,t_0)}^\theta(x)&=\Id_{\R^d},
    \end{aligned}
    \right. 
    \label{eq:linearized_ODE}
    \end{equation}
    where $Jf\colon\R^d\to\R^p$ denotes the Jacobian matrix of $f$.
\end{lemma}

\begin{proof} 
See \cite[Theorem 3.2.6]{bressan2007}.
\end{proof}
To evaluate the Lipscitz continuity of $\nabla_\theta\widetilde{J}$, we need to estimate the variation of $\Delta^\theta_{(0,t)}$ with respect to $\theta$.
This evaluation requires us to impose a further assumption on $v(x,\theta)=\theta f(x)$ in addition to \cref{assump:linear_NN}. In the following, we will denote $R^\ast(\norm{\theta})$ the same radius as in \cref{lem:cptspt}.
\begin{assumption}[Strong smoothness on $v$]\label{assump:additional}
The function $f$ is twice continuously differentiable, and for given $\theta\in L^2(0,T;\R^m)$ and $(x,y)\in B(R^\ast(\norm{\theta}))$, $f$ satisfies $\norm{f}_{C^1(\R^p;\R^d)}<\infty$.
\end{assumption}
This assumption correspond to \cite[Assumption 2]{bonnet2022measure}. Under this assumption, we can show the Lipschitz continuity by the same argument as in the proof of \cite[Lemma 3.1]{bonnet2022measure}.
\begin{lemma}[Fr\'echet-differentiablity of the loss functional]\label{lem:Frechet}
The sum of loss and kinetic regularization
\[
J_\ell\colon\theta\longmapsto\intRdY\ell\dd{\mu_T^\theta}+\frac{\lambda}{2}\int_0^T\intRd\abs{v(x,\theta)}^2\dd{\mu_t^\theta\dd t}
\]
is Fr\'echet-differentiable. In addition, for $\theta^1,\theta^2\in L^2(0,T;\R^m)$, there exists $C\qty(\lambda,\norm{\theta^1},\norm{\theta^2})>0$ such that
\[
    \norm{\nabla J_\ell(\theta^1)-\nabla J_\ell(\theta^2)}\leq C\qty(\lambda,\norm{\theta^1},\norm{\theta^2})\norm{\theta^1-\theta^2}.
\]
\end{lemma}
 From \cref{lem:Frechet}, the following corollary follows immediately.
 
\begin{corollary}[Semiconvexity for the parameter $\theta$]\label{cor:semiconv}
The functional
\begin{equation}
    \widetilde{J}\colon\theta\longmapsto J(\mu^\theta,\theta)=\int_{\mathbb{R}^d\times\Y}\ell\dd{\mu_T^\theta}+\int_0^T\int_{\mathbb{R}^d} 
    \qty({\frac{\lambda}{2}\abs{v\qty(x,\theta_t)}^{2
    }}+\frac{\epsilon}{2}\abs{\theta_t}^2)\dd{\mu_t^\theta(x)\dd t}\label{eq:Jtilde}    
\end{equation}
satisfies
\[
\widetilde{J}\qty(\qty(1-\zeta)\theta^1+\zeta\theta^2)\leq\qty(1-\zeta)\widetilde{J}\qty(\theta^1)+\zeta\widetilde{J}\qty(\theta^2)-\qty(\epsilon-C\qty(\lambda,\norm{\theta^1},\norm{\theta^2}))\frac{\zeta(1-\zeta)}{2}\norm{\theta^1-\theta^2}^2,
\]
for any $\theta^1,\theta^2\in L^2\qty(0,T;\R^m)$ and $\zeta\in\qty[0,1]$. Here $C\qty(\lambda,\norm{\theta^1},\norm{\theta^2})$ is the same positive number as in \cref{lem:Frechet}.
\end{corollary}

By this corollary, $\widetilde{J}$ is strongly convex on a $L^2$ ball if $\epsilon$ is sufficiently large compared to other parameters such as $\lambda$ and $T$. This plays an essential role in the proof of \cref{thm:existence_conv} below.

\begin{theorem}[Existence and uniqueness of  minimizer of $\widetilde{J}$]\label{thm:existence_conv}
Suppose that \cref{assump:linear_NN,assump:additional} If $\epsilon>0$ in \eqref{eq:Jtilde} is sufficiently large, there exists $\theta\in L^2(0,T;\R^m)$ which minimize $\widetilde{J}$, and $\theta$ is a unique minimizer of $\widetilde{J}$.
\end{theorem}

The proof is carried out using the direct method of Calculus of Variations as in \cref{thm:existence_of_kineticODE_Net}.

\begin{proof} 
It is clear that $0\leq\inf\widetilde{J}<+\infty$, then we can take a minimizing sequence $(\theta^n)_{n=1}^\infty\subset L^2(0,T;\R^m)$ such that $\widetilde{J}(\theta^n)\to\inf\widetilde{J}$ as $n\to\infty$. Thus, there exists $C>0$ independent of $n$ such that \eqref{eq:parameter_bound} holds. Then there exists a subsequence $(n')\subset\Z_{>0}$ such that \eqref{eq:theta_weak}. In addition, by \cref{cor:semiconv}, we see that there exists $\epsilon>0$ such that $\widetilde{J}$ is convex on $B(2C/\epsilon)$. Therefore, by Mazur's lemma, there exists another minimizing sequence $(\hat{\theta}^n)_{n=1}^\infty$ such that $\hat{\theta}_n\to\theta$ in $L^2(0,T;\R^m)$. Because $\widetilde{J}$ is lower semicontinuous, we conclude that
\[
    J(\theta)\leq\liminf_{n\to\infty}J(\hat{\theta}_n)=\inf_{L^2(0,T;\R^m)}\widetilde{J},
\]
i.e., $\theta$ is a minimizer of $\widetilde{J}$. The uniqueness is immediately obtained from the strong convexity of $\widetilde{J}$.
\end{proof}
\section*{Acknowledgement}
The author, N.I., would like to thank his supervisor, Norikazu Saito, for his encouragement and advice during the preparation of the paper.

\printbibliography
\end{document}